\documentclass[12pt]{amsart}

\numberwithin{equation}{section}

\usepackage{amsmath,amssymb,amsthm}
\usepackage[latin1]{inputenc}
\usepackage[english]{babel}
\usepackage{enumerate}

\usepackage{graphicx}

\usepackage[x11names,rgb]{xcolor}
\usepackage{tikz}
\usetikzlibrary{arrows,shapes}

\newtheorem{theorem}{Theorem}[section]
\newtheorem{corollary}[theorem]{Corollary}
\newtheorem{proposition}[theorem]{Proposition}
\newtheorem{lemma}[theorem]{Lemma}

\theoremstyle{definition}
\newtheorem{definition}[theorem]{Definition}
\newtheorem{example}[theorem]{Example}
\newtheorem{remark}[theorem]{Remark}

\numberwithin{equation}{section}

\newcommand{\Spec}{\textnormal{Spec}\,}
\newcommand{\Supp}{\mathrm{Supp}}
\newcommand{\cN}{\mathcal{N}}
\newcommand{\Ht}{\textnormal{Ht}}
\newcommand{\In}{\mathrm{in}}
\newcommand{\Aox}{A[\mathbf x,x_n]}
\newcommand{\Ax}{A[\mathbf x]}
\newcommand{\Kox}{K[\mathbf x,x_n]}
\newcommand{\Kx}{K[\mathbf x]}
\newcommand{\St}{\underline{\mathrm{St}}}

\def\bcr{\color{black}}
\def\ecr{\color{black}}
\def\bcR{\color{black}}

\title[]{The scheme of liftings and applications}

\author[]{C. Bertone}
\email{cristina.bertone@unito.it, margherita.roggero@unito.it}
\address{Dip. di Matematica dell'Universit\`a di Torino, Torino, Italy}

\author[]{F. Cioffi}
\email{cioffifr@unina.it, maguida@unina.it}
\address{Dip. di Matematica e Applicazioni dell'Universit\`a di Napoli, Napoli, Italy}

\author[]{M. Guida}

\author[]{M. Roggero}


\begin{document}

\begin{abstract}
We study the locus of the liftings of a homogeneous ideal $H$ in a polynomial ring over any field. We prove that this locus can be endowed with a structure of scheme $\mathrm L_H$ by applying the constructive methods of Gr\"obner bases, for any given term order. Indeed, this structure does not depend on
the term order, since it can be defined as the scheme representing the functor of liftings of $H$. We also provide an explicit isomorphism between the schemes corresponding to two different term orders.
 
Our approach allows to embed $\mathrm L_H$ in a Hilbert scheme as a locally closed subscheme, and, over an infinite field, leads to find  interesting  topological properties, as for instance that   $\mathrm L_H$ is connected and  that  its locus of radical liftings is open.  Moreover, we show that every ideal defining an arithmetically Cohen-Macaulay scheme of codimension two has a radical lifting, giving in particular an answer to an open question posed by L. G. Roberts in 1989.
\end{abstract}

\maketitle


\section*{Introduction}

In this paper we consider the lifting problem as proposed in terms of ideals first in \cite{GGR} and then in \cite{Ro} and, equivalently, in terms of $K$-algebras by Grothendieck (e.g. \cite{Ro} and the references therein or \cite{BuchEis}). Many authors have investigated this interesting problem, sometimes also describing particular lifting procedures to construct algebraic varieties with specific properties (see \cite{M,H66,GGR,Rlg,Ro,MiNa,LY} and the references therein). 

We propose and use a new approach that is based on the theory of representable functors. Indeed, we define {\em the functor of liftings} of a homogeneous polynomial ideal $H$ and show that it is representable, in the perspective given by \cite{LR2} for Gr\"obner strata and according to the point of view of 
\cite{BGS91}. Our approach is constructive and we compute {\em the scheme of liftings}   $\mathrm L_H$   of $H$, i.e.~the scheme that parameterises the liftings of $H$ and represents the functor, by a reformulation of a result of \cite{CFRo} in terms of Gr\"obner bases. 

An almost immediate result of the application of our approach, together with the features of Gr\"obner strata, is that $\mathrm L_H$  can be embedded in a Hilbert scheme, with the consequence that its locus of radical liftings  is an open subset.  This fact gives a contribution to a question posed in \cite[Remark at pag. 332]{LY}. 

Even though the scheme of liftings  $\mathrm L_H$ can be neither irreducible nor reduced (see Examples \ref{ex:LuoYilmaz} and \ref{es:struttNonRid}), we prove that, over an infinite field, $\mathrm L_H$ has several interesting topological properties. For instance, $\mathrm L_H$ is always connected, since the point  corresponding to $H$ belongs to every irreducible component of $\mathrm L_H$. Moreover, $\mathrm L_H$ is isomorphic to an affine space if and only if it is smooth at this point. These properties are proved exploiting the action of the torus $K^*=K\setminus \{0\}$ on $\mathrm L_H$.  

We then consider the special case of   ideals $H$   defining   arithmetically Cohen-Macaulay (aCM, for short) schemes of codimension two  and prove that  their schemes of liftings   are  isomorphic to   affine spaces. 
\bcR The problem of studying whether some particular  families   of ideals can be  parameterised by an affine space has been also treated by other authors. For instance, Gr\"obner strata of ideals defining aCM schemes in $\mathbb P^2$ are studied in \cite{CV,Co}, and other  Gr\"obner strata of  polynomial ideals in any number of  variables  are studied in \cite{RT} (see also the references therein). Some of the tools used in those papers also appear in the present one. In particular, we quote the action of the torus $K^*$ on the families of ideals, and, in the  aCM case, the description of ideals by means of Hilbert-Burch resolutions and the use of the Pommaret basis of quasi-stable ideals. \ecr

We are also able to prove that every saturated ideal defining an  aCM   scheme of codimension two has a radical lifting. This result is particularly significant in the context of the study of radical liftings, because of the lack of information endured until now 
in the case of polynomial homogeneous ideals in three variables. Indeed, we provide an affermative answer to the question posed by L. G. Roberts in \cite{Rlg}.

The paper is organized in the following way. 
Referring to \cite{FR,RT,LR,LR2}, in Section \ref{sec:grobner} we recall definitions and main features of Gr\"obner strata. Moreover, we give an improvement 
of \cite[Theorem 4.7]{LR} (Theorem \ref{drv}) which will be useful to embed the scheme of liftings of a homogeneous ideal in a Hilbert scheme. 

In Section \ref{sec:lifting functor}, we define the functor of liftings of a homogeneous polynomial ideal and introduce the constructive tool we use to represent it, i.e.~a reformulation of \cite[Theorem 2.5]{CFRo} by means of Gr\"obner bases (Theorem \ref{FM}). In Section \ref{sec:representability}, we prove that a functor of liftings is representable, thus obtaining that the construction of the scheme of liftings that arises from Theorem \ref{FM} does not depend on the given term order, up to isomorphisms (Theorem \ref{funtrappr1} and Corollary \ref{cor3}). In Section \ref{sec:constructions}, we give an explicit construction of the these isomorphisms (Theorem \ref{isomorfismo}).

In Section \ref{sec:torus}, we describe how we embed the scheme of liftings $\mathrm L_H$  in a Hilbert scheme and deduce that its  locus of radical liftings is an open subset (Proposition \ref{cor1} and Corollary \ref{cor4}). Then, we investigate the action of the torus on the scheme of liftings obtaining the topological properties we have previously described (Proposition \ref{prop:smooth} and Corollary \ref{cor2}). 

In Section \ref{sec:acm},  
we find that, if $H$ is a saturated homogeneous polynomial ideal defining an aCM scheme of codimension two, then  $\mathrm L_H$ is isomorphic to an affine space (Theorem \ref{th:spazio affine}). 
Moreover, exploiting the Hilbert-Burch Theorem and the potentiality of Gr\"obner deformations, we conceive a constructive method to show that every aCM scheme of codimension two has a radical lifting (Theorem \ref{thm:liftACMdue}).

All the results we present are based on constructive arguments. Hence,  
the last section is devoted to give explicative examples of the constructive methods we introduce and apply in this paper.


\section{Generalities}

A term  is a power product $x^\alpha = x_0^{\alpha_0}\cdot\ldots\cdot x_n^{\alpha_n}$.  Let $\mathbb T_{\mathbf x}$ and $\mathbb T_{\mathbf x,x_n}$ be the set of terms in the variables $\mathbf x=\lbrace x_0,\dots,x_{n-1}\rbrace$ and $\mathbf x,x_n=\{x_0,\ldots,x_{n-1},x_n\}$, respectively.
We assume that the variables are ordered as $x_0>x_1>\cdots>x_n$. 
The degree of a term is $\deg(x^\alpha)=\sum \alpha_i=\vert \alpha\vert$.

\begin{definition} For a given term order $\prec$ on $\mathbb T_{\mathbf x}$, we will denote by $\prec_n$ the corresponding {\em degreverse term order} in $\mathbb T_{\mathbf x,x_n}$, namely the graded term order such that for two terms $x^{\alpha}$ and $x^{\beta}$ in $\mathbb T_{\mathbf x,x_n}$ of the same degree,     $x^{\alpha}\prec_n x^{\beta}$ if $\alpha_n > \beta_n$ or $\alpha_n = \beta_n$ and $\frac{x^{\alpha}}{x_n^{\alpha_n}} \prec \frac{x^{\beta}}{x_n^{\beta_n}}$.
\end{definition}

We will always consider commutative rings with unit such that $1\neq 0$ and every morphism will preserve the unit. 

Let $K$ be a field. From now on, we will denote the polynomial ring $K[x_0,\dots,$ $x_{n-1}]$ by $K[\mathbf x]$ and the polynomial ring $K[x_0,\dots,x_n]$ 
by $K[\mathbf x,x_n]$. For any $K$-algebra $A$, $A[\mathbf x]$ will denote the polynomial ring $A\otimes_{K} K[\mathbf x]$ and $\Aox$ will denote 
$A\otimes_{K} K[\mathbf x,x_n]$. Obviously, $A[\mathbf x]$ is a subring of $\Aox$, hence the following notations and assumption will be stated for $\Aox$ but will hold for $A[\mathbf x]$ too. We assume that every $K$-algebra is Noetherian. 

We refer to \cite{KR1, SPES2} for standard facts about Gr\"obner bases. 
In the present paper, we only consider  either Gr\"obner bases in a polynomial ring over a field or \emph{monic} Gr\"obner bases in $A[\mathbf x,x_n]$. Hence, initial ideals are always generated by terms.

For any non-zero homogeneous polynomial $f \in \Aox$, the \textit{support} of $f$ is the set $\Supp(f)$ of terms in $\mathbb T_{\mathbf x,x_n}$ that appear in $f$ with a non-zero coefficient.  
The degree of $f$ is $\deg(f)=\max\{\deg(x^\alpha)\vert x^\alpha\in \Supp(f)\}$.
The \emph{head term} of a non-null polynomial $f$ is the maximum $\Ht(f)$ of $\Supp(f)$ w.r.t.~a given term order. 

\begin{remark}\label{rem:degreverse}
As for the graded reverse lexicographic term order ({\em degrevlex}, for short), which is a particular degreverse term order, also for every degreverse term 
order, we have that if the head term of a  homogeneous  polynomial $f$ is divisible by $x_n^r$ then the polynomial $f$ is divisible by $x_n^r$.
\end{remark}

A \emph{monomial ideal} is generated by terms. We denote by $\mathfrak j$ a monomial ideal in  $\Kx$  and by $J$ the monomial ideal $\mathfrak j\cdot K[\mathbf x,x_n]$. Note that $\mathfrak j$ and $J$ have the same monomial basis, that we denote by $B_{\mathfrak j}$.
We denote by $\cN(\mathfrak j)$ the \emph{sous-escalier} of $\mathfrak j$, that is the set of terms in $\mathbb T_{\mathbf x}$ not belonging to $\mathfrak j$. 
Analogously, we have $\cN(J)\subseteq \mathbb T_{\mathbf x,x_n}$.


\section{Background: Gr\"obner strata}
\label{sec:grobner}

In this section, we recall some results about families of homogeneous polynomial ideals sharing the same initial ideal with respect to a given term order (see Definition \ref{defGS}), and we call them Gr\"obner strata, as in \cite{Led, LR}. In other papers they are called and denoted in several different ways (e.g., \cite{CV,Co, RT} and the references therein). 

Here,  we are interested in the scheme-theoretic point of view, rather than in a set-theoretic study of Gr\"obner strata. Although this point of view underlies several papers (e.g. \cite{CF,NS}), in \cite{LR} we find the first proof of the fact that a Gr\"obner stratum can be endowed with an affine scheme structure that {\em does not depend on the reduction procedure} which is applied to compute it (see \cite[Proposition 3.5]{LR}). This structure is possibly non-reduced.

\medskip
In what follows, the polynomials and ideals we consider will always be homogeneous
with respect to the variables $\mathbf x,x_n$. 
Moreover, the polynomial ideals that we consider have and are always given by monic Gr\"obner bases.  This is a key point for the use of functors we will introduce, because of the following two facts:
\begin{enumerate}[(i)]
\item there is no ambiguity when using the terminology \lq\lq initial ideal\rq\rq, because the head terms  of the polynomials in the  Gr\"obner bases we consider have coefficient 1;
\item if $\varphi \colon A \rightarrow B$ is a morphism of $K$-algebras and $I$ is an ideal in $A[\mathbf x]$ (or $\Aox$) generated by a {\em monic} 
Gr\"obner basis $G_I$, then $I\otimes_A B$ is generated by $\varphi(G_I)$ which is again a {\em monic} Gr\"obner basis, with the same head terms \cite{BGS91}. 
In other words, monic Gr\"obner bases have a good behavior with respect to the extension of scalars. Recall that the polynomials of a reduced Gr\"obner basis 
are monic by definition. 
\end{enumerate}

Now, let $J$ be any monomial ideal in $\Kox$ and $\sigma$  be   a term order on $\mathbb T_{\mathbf x,x_n}$. Given an ideal $I$ in $\Aox$, we will denote by
$\mathrm{in}_{\sigma}(I)$ the initial ideal of $I$ w.r.t.~$\sigma$.

\begin{definition} \label{defGS} \cite{LR} 
The family of the homogeneous ideals $I\subseteq \Aox$ with  $\mathrm{in}_{\sigma}(I)=J\otimes_K A$ is called a {\em Gr\"obner stratum} and denoted by $\St_J^\sigma(A)$.
\end{definition}

By construction, the ideals belonging to a Gr\"obner stratum share the same Hilbert function. Further, $\St_J^\sigma$ is a representable functor between the
category of Noetherian $K$-algebras and that of sets. We call its representing scheme {\em Gr\"obner stratum scheme} and denote it by $\mathrm{St}_J^\sigma$.
Now, we  briefly recall the construction of the representing scheme $\mathrm{St}_J^\sigma$ and some of its main features, for which anyway we refer to \cite{LR,LR2}.

In order to compute $\mathrm{St}_J^\sigma$, we consider a set of polynomials $\mathcal G$ of the following shape:
\begin{equation}\label{JbaseC} \mathcal{G}= \{F_\alpha=x^\alpha +  \sum C_{\alpha\gamma} x^\gamma :
\Ht(F_\alpha)=x^\alpha\in B_J\}\subset K[C][\mathbf x,x_n]
\end{equation}
where the summation runs over the set $\{x^\gamma \in \mathbb T_{\mathbf x,x_n}: \vert\gamma\vert=\vert\alpha\vert \text{ and }x^\gamma\sigma x^\alpha\}$ and $C$ is a compact notation for the set of new variables $C_{\alpha \gamma}$.

Denote by $\mathfrak{a}$ the ideal in $K[C]$ generated by the coefficients of the terms of $\mathbb T_{\mathbf x,x_n}$ in a complete reductions by $\mathcal G$ of the $S$-polynomials $S(F_\alpha,F_{\beta})$ with respect to $\sigma$. By \cite[Proposition 3.5]{LR}, the ideal $\mathfrak{a}$ depends only on $J$ and $\sigma$, because it can be defined in an
equivalent intrinsic way, and defines the affine scheme $\mathrm{St}_J^\sigma$.

If, in particular, $J$ is a strongly stable saturated ideal and $\sigma$ is the degrevlex term order, then we have 
\begin{equation}\label{isomorfismo grobner}
\mathrm{St}_J^\sigma\simeq \mathrm{St}_{J_{\geq m}}^\sigma,
\end{equation}
for every positive integer $m$ \cite[Proposition 4.11]{LR}.
This last result holds under weaker hypotheses on $J$ and on $\sigma$, as now we prove.

\begin{theorem}\label{drv}
Let $J\subset \Kox$ be a monomial ideal with $B_J \subset \mathbb T_{\mathbf x}$ and let $\prec_n$ be a degreverse term order.  Then
$$\mathrm{St}_J^{\prec_n}\cong \mathrm{St}_{J_{\geq m}}^{\prec_n}, \text{ for every integer } m,$$
and $\mathrm{St}_J^{\prec_n}$ can be embedded in the Hilbert scheme  $\mathcal{H}ilb^n_{p(t)}$ as a locally closed subscheme, where $p(t)$ is the Hilbert polynomial of $\Kox/J$.
\end{theorem}

\begin{proof}
We show that the functors $\St_{J_{\geq m}}^{\prec_n}$ and $\St_{J}^{\prec_n}$ are isomorphic. 
As a consequence, their representing schemes are isomorphic too. We denote by $G_I=\{f_\alpha\}_\alpha$ the reduced Gr\"obner basis of $I\in \St_{J}^{\prec_n} (A)$. Then, the ideal $I_{\geq m}$ has the Gr\"obner basis $W_{I_{\geq m}}$, that consists of the polynomials $f_\alpha x^\gamma$ with $\vert \gamma \vert =d_\alpha:=\max\{0,m-\vert\alpha\vert \}$. Observe that the polynomials in $W_{I_{\geq m}}$ are still monic also if in general $W_{I_{\geq m}}$ is not reduced.
Thus, the following well-defined map is a natural transformation of functors
\[\begin{array}{rcl}
 \St_J^{\prec_n} &\rightarrow & \St_{J_{\geq m}}^{\prec_n}\\ 
I &\mapsto& I_{\geq m}.
\end{array}
\] 
Indeed, for every $K$-algebra morphism  $\varphi \colon A\rightarrow  B$ and ideal $I \in \St_J^{\prec_n}(A)$,  we obtain $I_{\geq m}\otimes_A B=(I\otimes_A B)_{\geq m}$ because both these two ideals are generated by $\varphi (W_{I_{\geq m}})$ in $B[\mathbf x, x_n]$.  
The above natural transformation of functors is actually an isomorphism, with inverse
\[\begin{array}{rcl}
\St_{J_{\geq m}}^{\prec_n}& \rightarrow&  \St_{J}^{\prec_n}\\
L &\mapsto& (L \colon x_n^\infty).
\end{array}\] 
Indeed, consider the reduced Gr\"obner basis of an ideal  $L\in \St_{J_{\geq m}}^{\prec_n}$: for every term $x^\alpha \in B_J$, this basis contains a polynomial
$h_\alpha$ whose head term is $x^\alpha x_n^{d_\alpha}$. Since $\prec_n$ is a degreverse term order, we have 
$h_\alpha/x_n^{d_\alpha}\in (L \colon x_n^\infty)$, hence
$(L \colon x_n^\infty)$ contains the set of generators of an ideal whose initial ideal is $J$.  Since every  element in the basis of $J$ is not divisible by $x_n$, we have $\mathrm{in}_{\prec_n}(L \colon x_n^\infty) = J$ and, hence, $(L \colon x_n^\infty)\in \St_J^{\prec_n}(A)$.

Arguing again on those polynomials in $L$ that are monic and with head terms of kind $x^\alpha x_n^{d_\alpha}$ for every $x^\alpha \in B_J$, we get that for every $K$-algebra morphism $\varphi \colon A\rightarrow B$ we obtain $((L\otimes_A B) \colon x_n^\infty)=(L\colon x_n^\infty)\otimes_A B \in \St_{J}^\prec (B)$, where the inclusion \lq\lq$\supseteq$\rq\rq\ is a standard fact (e.g. \cite[Exercise 1.18]{AM}). 
For the other inclusion, let $f$ be a monic polynomial in $L$ and consider $f\otimes_A b\in L\otimes_A B$. Let  $x_n^k$ be the maximal power of $x_n$ by which the head term of $f$ is divisible. Then, $x_n^k$ is also the maximal power of $x_n$ by which $f$ is divisible (see Remark \ref{rem:degreverse}) and the same happens for $f\otimes_A b$, because the extension
of scalars does not modify the head terms. In conclusion, letting $\bar f:=f/x_n^k$, we have that $\bar f\otimes_A b$ belongs to $((L\otimes_A B) \colon x_n^\infty)$, thus it belongs to $(L\colon x_n^\infty)\otimes_A B$.

The last assertion is a consequence of the previous one and of \cite[Theorem 6.3]{LR}.
\end{proof}


\section{The functor of liftings of a homogeneous polynomial ideal}
\label{sec:lifting functor}

In this section, first we recall what a lifting of a given homogeneous ideal $H\subseteq \Kx$ with respect to $x_n$ is, referring to \cite{GGR,Ro,LY}. Then, following the perspective of \cite{LR2}, we introduce a functorial description of these liftings.

\begin{definition}\label{def:lifting}
Let $H$ be a homogeneous ideal of $\Kx$ and $A$ be a Noetherian $K$-algebra. A homogeneous ideal $I$ of $\Aox$ is called a {\em lifting of $H$ with respect to $x_n$}
or a {\em $x_n$-lifting of $H$}
if the following conditions are satisfied:
\begin{enumerate}
\item[(a)] the indeterminate $x_n$ is not a zero-divisor in $\Aox/I$;
\item[(b)] $(I,x_n)/(x_n)\simeq H\Ax$ under the canonical isomorphism $\Aox/(x_n)\simeq \Ax$;
 
or, equivalently,
\item[(b$'$)] { $\{g(x_0,x_1,\ldots,x_{n-1},0) : g\in I\}=H\Ax$}.
\end{enumerate}
\end{definition}

By the definition, a $x_n$-lifting is a saturated ideal. 
For every homogeneous ideal $H\subseteq \Kx$ and for every $K$-algebra $A$,  consider the set 
\begin{equation*}
\underline{\mathrm{L_{H}}}(A)=\lbrace I\subseteq \Aox : I \text{ is a $x_n$-lifting of }H \rbrace.
\end{equation*}
Next result is a reformulation of \cite[Theorem 2.5]{CFRo} (see also \cite[Proposition 6.2.6]{KR2}) in terms of Gr\"obner bases.  

\begin{theorem} \label{FM}
Let $A$ be a $K$-algebra, $H$ a homogeneous ideal of $\Kx$ and $I$ a homogeneous ideal of $\Aox$. Then, the following conditions are equivalent:
\begin{enumerate}[(i)]
\item\label{FM_i} the ideal $I$ belongs to $\underline{\mathrm{L_{H}}}(A)$;
\item\label{FM_ii} the reduced Gr\"obner basis of $I$ w.r.t.~a degreverse term order $\prec_n$ on $\Aox$ is $\lbrace f_\alpha+g_\alpha\rbrace_\alpha$, 
where $\lbrace f_\alpha\rbrace_\alpha$ is the reduced Gr\"obner basis of $H$ w.r.t.~$\prec$ and $g_\alpha \in (x_n)\Aox$.
\end{enumerate}
{Furthermore, if $I\subset \Aox$ is an $x_n$-lifting of $H$, then   $\mathrm{in}_{\prec_n}(I)$ is generated by the same terms as $\mathrm{in}_{\prec}(H)$.}
\end{theorem}

\begin{proof} 
To prove that $(ii)$ implies $(i)$ first we observe that $x_n$ is not a zero-divisor in $\Aox/I$, because the head terms of the polynomials $f_\alpha+g_\alpha$ are not divisible by $x_n$ and these polynomials form a Gr\"obner basis w.r.t.~$\prec_n$. Moreover, for every $\alpha$, $(f_\alpha+g_\alpha)(x_0,\dots,x_{n-1},0)=f_\alpha$, hence  $\{g(x_0,x_1,\ldots,x_{n-1},0) :\forall g\in I\}=H\Aox$.

To prove that $(i)$ implies $(ii)$, observe that if $I$ is a $x_n$-lifting of $H$, then $x_n$ is not a zero-divisor in $\Aox/I$. {Thus, the head terms of the polynomials of the reduced Gr\"obner basis of $I$ w.r.t.~$\prec_n$ are not divisible by $x_n$, by Remark \ref{rem:degreverse}. Then, we conclude the proof applying condition (b$'$) of Definition \ref{def:lifting} from which we deduce that $\mathrm{in}_{\prec_n}(I)$ and $\mathrm{in}_{\prec}(H)$ are generated by the same set of terms.}
\end{proof}


If $\phi:A\rightarrow B$ is a $K$-algebra morphism, we denote by $\phi$ also the natural extension of $\phi$ to $\Aox$ and recall that the image under $\phi$ of every ideal $I$ in $\Aox$ generates the extension  $I^e = IB[\mathbf x,x_n]=I \otimes_A B$ (see \cite{BGS91}). 

\begin{corollary} 
If $\phi:A\rightarrow B$ is a $K$-algebra morphism, then for every $I \in \underline{\mathrm{L}}_H(A)$ the ideal $I\otimes_A B$ belongs to $\underline{\mathrm{L}}_H(B)$.
\end{corollary}

\begin{proof}
Let $G$ be the reduced Gr\"obner basis of $H$ with respect to a term order $\prec$ and $G_I$ be the  reduced Gr\"obner basis of $I\in \underline{\mathrm{L}}_H(A)$ w.r.t.~$\prec_n$. Thus, by Theorem \ref{FM} we have
\[
G=\lbrace f_\alpha\rbrace_\alpha, \quad G_I=\lbrace f_\alpha+g_\alpha\rbrace_\alpha=\lbrace f_\alpha+\sum  c_{\alpha\gamma}x_nx^\gamma\rbrace_\alpha, \quad c_{\alpha\gamma} \in A,
\]
where the summation runs over the set $\{x_nx^\gamma\in \cN(J): \deg(x_nx^\gamma)=\deg(f_\alpha)\}$.

The ideal $I\otimes_A B$ is then generated by 
$\phi(G_I)=\lbrace  f_\alpha+\sum \phi(c_{\alpha\gamma})
x_nx^\gamma\rbrace_\alpha$, which is still a reduced Gr\"obner basis with respect to $\prec_n$ because the polynomials of $G_I$ are monic. Hence, $I\otimes_A B$ is a $x_n$-lifting of $H$ in $B[\mathbf x,x_n]$ by Theorem \ref{FM}.
\end{proof}

The  previous result allows us to  define  a functor. 

\begin{definition}  The  {\em functor of liftings}  of $H$ 
\[
\underline{\mathrm{L_{H}}}: \underline{\text{Noeth-}K\text{-Alg}}\rightarrow \underline{\mathrm{Set}}
\]
associates to every Noetherian $K$-algebra $A$ the set $\underline{\mathrm{L_{H}}}(A)$
and to every morphism of $K$-algebras $\phi:A \rightarrow B$ the map
\[\begin{array}{rcl}
\underline{\mathrm{L_{H}}}(\phi): \underline{\mathrm{L_{H}}}(A)&\rightarrow& \underline{\mathrm{L_{H}}}(B)\\
 I&\mapsto& I\otimes_A B.
\end{array}\]
\end{definition}

\begin{remark} \label{piugenerale}
In \cite{Ro}, the definition of $x_n$-lifting is actually given by a more general version of condition (b$'$). Indeed, the ideal $H$ can be replaced by 
its image via an automorphism  $\theta$  
of $\Kx$ as a graded $K$-algebra. Anyway, our study of $x_n$-liftings by the functor $\underline{\mathrm{L_{H}}}$ includes this more general situation, as we will see in Section \ref{sec:torus}. 
%
\end{remark}


\section{Representability of the functor  $\underline{\mathrm{L_{H}}}$}
\label{sec:representability}

Given the homogeneous ideal $H$ in $\Kx$ and its reduced Gr\"obner basis $G=\lbrace f_\alpha\rbrace_\alpha$ w.r.t.~$\prec$, for every $x^\alpha \in \In_\prec(H)$ we define
\begin{equation}\label{code}
g_\alpha:=\sum_{x_nx^\gamma\in \cN(J)_{\vert\alpha\vert}}C_{\alpha\gamma}x_nx^\gamma, \quad \mathcal G=\lbrace f_\alpha+g_\alpha\rbrace_\alpha,
\end{equation}
where the $C_{\alpha\gamma}$'s are new variables. We set  $C=\lbrace C_{\alpha\gamma}\rbrace_{\alpha,\gamma}$ and give a term order    on the terms of $K[C]$ by which we extend the term order $\prec_n$ to an elimination term order of the variables $\mathbf x,x_n$ in $K[C][\mathbf x,x_n]$. For simplicity, we keep on using the notation $\prec_n$ for this elimination term order on $K[C][\mathbf x,x_n]$. As usual, we denote by $S(f_\alpha+g_\alpha,f_{\alpha'}+g_{\alpha'})$ the $S$-polynomial between $f_\alpha+g_\alpha$ and $f_{\alpha'}+g_{\alpha'}$.


\begin{proposition}\label{comeLR}
Let $H$, $G$, $\mathcal G$, $\prec$ and $\prec_n$ be as above and consider
an ideal $\mathfrak h$ in $K[C]$ with Gr\"obner basis $\mathcal H$. The followings are equivalent:
\begin{enumerate}[(i)]
\item \label{comeLR_i}$\mathcal G\cup \mathcal H$ is a Gr\"obner basis in $K[C][\mathbf x,x_n]$;
\item \label{comeLR_ii}$\mathfrak h$ contains the coefficients of the terms of $\mathbb T_{\mathbf x,x_n}$ in all the polynomials in the ideal $(\mathcal G)K[C][\mathbf x,x_n]$ that are reduced modulo
$\In_\prec(G)$;
\item \label{comeLR_iii} $\mathfrak h$ contains the coefficients of the terms of $\mathbb T_{\mathbf x,x_n}$ in every complete reduction by $\mathcal G$ of $S(f_\alpha+g_\alpha,f_{\alpha'}+g_{\alpha'})$ with respect  to $\prec_n$, for every $\alpha$ and $\alpha'$;
{\item \label{comeLR_iv}$\mathfrak h$ contains all the coefficients of the terms of $\mathbb T_{\mathbf x,x_n}$ in a complete reduction by $\mathcal G$ of $S(f_\alpha+g_\alpha,f_{\alpha'}+g_{\alpha'})$ with respect  to $\prec_n$, for every $\alpha$ and $\alpha'$.
}
\end{enumerate}
\end{proposition}

\begin{proof}
It is sufficient to repeat the arguments of \cite[Proposition 3.5]{LR}.
\end{proof}

{
\begin{definition}\label{defh0}
We denote by $\mathfrak h_0$ the ideal in $K[C]$ generated by the coefficients of the terms of $\mathbb T_{\mathbf x,x_n}$ in a complete reduction with respect to $\mathcal G$ of $S(f_\alpha+g_\alpha,f_{\alpha'}+g_{\alpha'})$, for every $\alpha$ and $\alpha'$.
\end{definition}

At first glance, one might think that the ideal $\mathfrak h_0$ of Definition \ref{defh0} is not unique. In fact, in general, $\mathcal G$  is not a Gr\"obner basis, so that there could be several different complete reductions of each  polynomial $S(f_\alpha+g_\alpha,f_{\alpha'}+g_{\alpha'})$, hence several different sets of polynomials in $K[C]$  generating  ideals fulfilling Definition \ref{defh0}. Nevertheless, exploiting  the equivalence between the conditions of Proposition \ref{comeLR},   we  see that all these different  sets    generate  the same ideal $\mathfrak h_0$. Indeed, by definition,  $\mathfrak h_0$  satisfies the condition  (iv) of Proposition \ref{comeLR}, so that it also satisfies condition (iii). Moreover, by definition, $\mathfrak h_0$ is contained in every ideal $\mathfrak h$  that satisfies condition (iii).  Thus, we could define $\mathfrak h_0$  in an  intrinsic way, for instance as the intersection of the ideals that satisfy the conditions of Proposition \ref{comeLR} or as the minimum with respect to the inclusion among them. 
}

\begin{theorem}\label{funtrappr1}
Let $H\subseteq \Kx$ be a homogeneous ideal and $\prec$ be a term order. The affine scheme $\mathrm{Spec}(K[C]/\mathfrak h_0)$ represents the functor $\underline{\mathrm L}_H$.
\end{theorem}

\begin{proof}
By Theorem \ref{FM}, an ideal $I$ in $\Aox$ belongs to $\underline{\mathrm L}_H(A)$ if and only if it has a reduced Gr\"obner basis with respect to $\prec_n$ of the shape of \eqref{code}. By Proposition \ref{comeLR} and by construction, we get that $\mathcal G$ becomes a reduced Gr\"obner basis if and only if the parameters $C_{\alpha\gamma}$ are replaced by constants $c_{\alpha\gamma}\in A$ that satisfy the conditions in $\mathfrak h_0$.
For a given $I\in \underline{\mathrm L}_H(A)$, there is a unique choice of the coefficients $c_{\alpha\gamma}$ giving rise to the unique reduced Gr\"obner basis of $I$, and this choice corresponds to a $K$-algebra morphism $K[C]/\mathfrak h_0\rightarrow A$, i.e.~to a scheme morphism $\mathrm{Spec}(A)\rightarrow \mathrm{Spec} (K[C]/\mathfrak h_0)$.
\end{proof}

\begin{definition}
For every homogeneous ideal $H\subset \Kx$, the representing scheme of the functor $\underline{\mathrm L}_H$ is called the \emph{scheme of liftings of $H$} and is denoted by $\mathrm L_H$.
\end{definition}

{
As an immediate consequence of our functorial approach, we also obtain that $\mathrm L_H$ is independent of the term order. We single out this result because, as we highlighted at the beginning of Section 2, we are interested in the scheme structure of $\mathrm{Spec} (K[C]/\mathfrak h_0)$, not only in the set of points at which the polynomials of $\mathfrak h_0$ vanish  (see Example \ref{es:struttNonRid}). 
}

\begin{corollary}\label{cor3}
For every homogeneous ideal $H\subset \Kx$, the scheme of liftings $\mathrm L_H$ is independent of the term order on $\Kx$ that is used to construct it. 
\end{corollary}

\begin{proof}
For every term order on $\Kx$, the scheme we obtain by Theorem \ref{funtrappr1} represents $\underline{\mathrm{L}}_H$. By Yoneda Lemma such a scheme is unique, up to isomorphisms.
\end{proof}


\section{Explicit construction of the isomorphisms}
\label{sec:constructions}

The functorial approach used in Sections \ref{sec:lifting functor} and \ref{sec:representability} immediately led to Corollary \ref{cor3}. However, if we consider two different term orders $\prec$ and $\prec'$ on $\Kx$, the functorial approach does not provide the isomorphism between the two schemes of Theorem \ref{funtrappr1} that we obtain by Proposition \ref{comeLR} starting from the two term orders.

Now, we explicitly construct this isomorphism. 
Let $B$ and $B'$ be the minimal sets of terms generating the initial ideals $\mathfrak j:=\In_\prec(H)$ and $\mathfrak j':=\In_{\prec'}(H)$, respectively, and let
\begin{equation*}
G:=\{ f_\alpha \ \vert \ \Ht(f_\alpha)=x^\alpha \in B \}  \hbox{ and } G':=\{ f_\beta' \ \vert \ \Ht(f_\beta)=x^\beta \in B' \}
\end{equation*}
be the reduced Gr\"obner bases of $H$ w.r.t.~$\prec$ and $\prec'$, respectively. 
Then, we have 
\begin{equation}\label{relazioni}
f_\beta'=\sum h_{\alpha \beta} f_\alpha \ \ , \ \ \forall f_\beta'\in G'.
\end{equation}
By a given term order on the terms of $K[C]$, we extend both the term orders $\prec_n$ and $\prec'_n$ to elimination term orders of the variables $\mathbf x,x_n$ on $K[C][\mathbf x,x_n]$. For simplicity, we keep on using the notations $\prec_n$ and $\prec'_n$ for these eliminations term order on $K[C][\mathbf x,x_n]$. Note that these term orders coincide when they are restricted to $K[C]$.

We define $\mathcal G:=\{f_\alpha+g_\alpha\}$, where $g_\alpha:=\sum_{{x_nx^\gamma\in \cN(J)_{\vert \alpha\vert}}}C_{\alpha\gamma}x_nx^\gamma$, and 
$\mathcal G':=\{f'_\beta+g'_\beta\}$, where $g'_\beta:=\sum_{{x_nx^\eta\in \cN(J')_{\vert\beta\vert}}}C'_{\beta\eta}x_nx^\eta$.

Recall that the set of initial terms of $\mathcal G$ with respect to $\prec_n$ is exactly $B$ and the set of initial terms of $\mathcal G'$ with respect to $\prec'_n$ is exactly $B'$: this is due to the fact that the terms in the supports of $g_\alpha$ and $g'_\beta$ are divisible by $x_n$, by construction.

Let $\mathcal{H}\subset K[C]$ be the reduced Gr\"obner basis w.r.t.~$\prec_n$ of the ideal $\mathfrak h_0$ of Theorem \ref{funtrappr1}. Thus, by Proposition \ref{comeLR}, $\mathcal G\cup\mathcal H$ is a Gr\"obner basis w.r.t.~$\prec_n$ in $K[C][\mathbf x,x_n]$.

Using the relations  \eqref{relazioni}, for every $x^{\beta}\in B'$ we define
$$p_\beta' :=\sum h_{\alpha \beta} (f_\alpha+g_\alpha)= f_{\beta}'+ \sum h_{\alpha \beta} g_\alpha \in K[C][\mathbf x,x_n].$$

\begin{lemma}\label{lemma 5.1}
The set of polynomials $\lbrace p'_\beta\rbrace_\beta\cup \mathcal H\subseteq K[C][\mathbf x,x_n]$ is a monic Gr\"obner basis w.r.t.~$\prec'_n$ of $(\mathcal G\cup \mathcal H)$.
\end{lemma}

\begin{proof}
By construction, we have $\lbrace p_\beta'\rbrace_\beta \cup \mathcal H\subseteq (\mathcal G\cup \mathcal H)$, hence $(\lbrace p_\beta'\rbrace_\beta \cup 
\mathcal H)\subseteq (\mathcal G\cup \mathcal H)$. We now prove the other inclusion. Further, we prove that $\lbrace p_\beta'\rbrace \cup \mathcal H$ is a Gr\"obner
basis of $(\mathcal G\cup \mathcal H)$ w.r.t.~$\prec'_n$, showing that every homogeneous polynomial $f$ in $(\mathcal G\cup \mathcal H)$ can be reduced to 0 by 
$\lbrace p_\beta'\rbrace \cup \mathcal H$ using $\prec_n'$. We  proceed by induction on the degree of $f$ w.r.t.~the variables $\mathbf x,x_n$.

The $0$-degree modules $(\lbrace p_\beta'\rbrace \cup \mathcal H)_0$ and $(\mathcal G\cup \mathcal H)_0$ are both equal to $\mathcal H$, which is a Gr\"obner
basis w.r.t.~$\prec_n$ and $\prec_n'$, because $\prec_n$ and $\prec_n'$ coincide on monomials of $K[C]$. We now assume $(\lbrace p_\beta'\rbrace \cup 
\mathcal H)_{m-1}=(\mathcal G\cup \mathcal H)_{m-1}$ and prove  $(\lbrace p_\beta'\rbrace \cup \mathcal H)_{m}=(\mathcal G\cup \mathcal H)_{m}$.

We consider $f\in (\mathcal G\cup \mathcal H)_m$, and  start the reduction by using the polynomials in $\lbrace p_\beta'\rbrace$:
\[
f\xrightarrow{\lbrace p'_\beta\rbrace}p, \quad \text{with } \Supp(p)\subseteq \mathcal N(J') \text{ and } p=x_n\cdot p'.
\]
By the construction of the polynomials $p_\beta'$, the polynomial $p$ belongs to $(\mathcal G \cup \mathcal H)_m$ and $x_n$ is not a $0$-divisor on $K[C][\mathbf x, x_n]/(\mathcal G\cup\mathcal H)$: hence, $p'$ belongs to $(\mathcal G \cup \mathcal H)_{m-1}$. By the inductive hypothesis, we have $p'\xrightarrow{\lbrace p'_\beta\rbrace\cup \mathcal H}0$, more precisely $p'$ is reduced to $0$ by $\mathcal H$, because $\Supp(p')\subseteq 
\mathcal N(J')$. Hence, we get $p\xrightarrow{\mathcal H}0$, as desired.
\end{proof}

We consider the unique reduced  Gr\"obner basis we obtain from $\{p'_\beta\}\cup \mathcal H$ by interreducing the polynomials $p'_\beta$ and denote by
$q_\beta'$ the reduced forms of the polynomials $p'_\beta$:
$$q'_\beta=f_{\beta}'+ x_n m_{ \beta}, \quad \text{with }\Supp(m_\beta)\subset \mathcal N(J'). $$

By comparing the polynomials $q_{\beta}'$ and $f_\beta'+g_\beta'$ we obtain a $K$-algebra morphism
\begin{equation}\label{iso}
\phi \colon K[C'][\mathbf x,x_n]\rightarrow K[C][\mathbf x,x_n]
\end{equation}
such that $\phi(x^\gamma)=x^\gamma$ and  $\phi(C'_{\beta\gamma})$ is the coefficient of $x^\gamma$ in $q'_\beta$. 

Take a term order on $K[C']$ and extend $\prec_n$ and $\prec'_n$ to $K[C'][\mathbf x,x_n]$ as done previously for $K[C][\mathbf x,x_n]$. 
Let $\mathcal{H'}\subset K[C']$ be the reduced Gr\"obner basis with respect to $\prec'_n$ of the ideal $\mathfrak h'_0$ as in Theorem \ref{funtrappr1}. Thus, by Proposition \ref{comeLR}, $\mathcal G'\cup\mathcal H'$ is a Gr\"obner basis w.r.t.~$\prec'_n$ in $K[C'][\mathbf x,x_n]$. 

\begin{theorem}\label{isomorfismo}
The $K$-algebra morphism $\phi$ of \eqref{iso} induces an isomorphism between $K[C']/\mathfrak h'_0$ and $K[C]/\mathfrak h_0$.
\end{theorem}
\begin{proof}

We first observe that $\phi(\mathcal H')\subseteq (\mathcal H)$. Indeed, by Proposition \ref{comeLR}(ii), $(\mathcal H')$ contains the coefficients of the terms of $\mathbb T_{\mathbf x,x_n}$ of every polynomial of kind $\sum l_\beta(f'_\beta+g'_\beta)$ with $\Supp(\sum l_\beta(f'_\beta+g'_\beta))\subseteq \mathcal N(J')$
and it is sufficient to observe 
\[
\phi(\sum l_\beta(f'_\beta+g'_\beta))=\sum  \phi(l_\beta)q'_\beta, \quad \Supp\left(\sum \phi(l_\beta) q'_\beta\right)\subseteq \mathcal N(J').
\]
Hence, by Proposition \ref{comeLR} and Lemma \ref{lemma 5.1}, the coefficients of $\phi(\sum l_\beta(f'_\beta+g'_\beta))$ belong to $(\mathcal H)$.
Repeating the construction of $\phi$ starting now from the relations $f_\alpha=\sum h'_{\alpha\beta}f_\beta'$, we obtain a $K$-algebra homomorphism 
$$\psi \colon K[C][\mathbf x,x_n]\rightarrow K[C'][\mathbf x,x_n]$$
such that $\psi(\mathcal H)\subseteq (\mathcal H')$.  Hence, it makes sense to restrict $\phi$ and $\psi$ to $K[C']/(\mathcal H')$ and $K[C]/(\mathcal H)$. 

We define $\chi:=\phi\circ\psi:K[C][\mathbf x,x_n]\rightarrow K[C][\mathbf x,x_n]$. We finally show that its restriction $\overline\chi:K[C]/(\mathcal H)
\rightarrow K[C]/(\mathcal H)$ 
is the identity.

It is sufficient to observe that the polynomial $C_{\alpha\gamma}-\chi(C_{\alpha\gamma})$ belongs to $(\mathcal H)$ for every $\alpha$ and every $\gamma$. 
Indeed, $\chi(f_\alpha+g_\alpha)=f_\alpha+\chi(g_\alpha)$. By construction of $\phi$ and $\psi$, $f_\alpha+\chi(g_\alpha)$ belongs to $(\mathcal
G\cup \mathcal H)$.
In particular, $f_\alpha+g_\alpha-(f_\alpha+\chi(g_\alpha))=g_\alpha-\chi(g_\alpha)$ is a polynomial in $(\mathcal G\cup \mathcal H)$ {with 
$\Supp(g_\alpha-\chi(g_\alpha))\subset\mathcal N(J)$}. Hence, by Proposition \ref{comeLR}, the coefficients of the terms of $\mathbb T_{\mathbf x,x_n}$ of $g_\alpha-\chi(g_\alpha)$ belong to $(\mathcal H)$. Then, in $K[C]/(\mathcal H)$ we have $C_{\alpha\gamma}=\overline\chi(C_{\alpha\gamma})$, hence $\overline\chi$ is the identity on $K[C]/(\mathcal H)$.
\end{proof}

In Example \ref{esiso} we will exhibit an explicit computation of the isomorphism described in Theorem \ref{isomorfismo}. We observe that this isomorphism 
$K[C]/\mathfrak h_0\simeq K[C']/\mathfrak h'_0$ is constructed from  $\phi:K[C]\rightarrow K[C']$, which in general is not an isomorphism.


\section{The torus action on  the scheme of liftings}
\label{sec:torus}

It is obvious by the definition that the $x_n$-liftings of an ideal $H\subset \Kx$ have the same Hilbert function, then also the same Hilbert polynomial. Hence, they define points of a same Hilbert scheme. In this section, exploiting the relation with the Hilbert scheme, we will obtain several interesting properties of the scheme of liftings. In particular, we investigate  an action of the  torus $K^*:=K\setminus \{0\}$   on $\mathrm L_H$.

{\bf From now on, we assume that the ground field $K$ is infinite.}

\begin{proposition}\label{cor1}
Consider $J=\mathrm{in}_\prec(H)K[\mathbf x,x_n]$ and let $\overline p(t)$ be the Hilbert polynomial of $\Kox/J$. Then, $\underline{\mathrm L}_H$ is a closed subfunctor of $\underline{\mathrm{St}}_J^{\prec_n}$ and $\mathrm L_H$ is a locally closed subscheme of the Hilbert scheme $\mathcal{H}ilb^n_{\overline p(t)}$.
\end{proposition}

\begin{proof}
We compute the ideal $\mathfrak a$ defining $\mathrm{St}_J^{\prec_n}$ and the reduced Gr\"obner basis $G$ of $H$. Then, we consider the linear section of $\mathrm{St}_J^{\prec_n}$ obtained by the polynomials $C_{\alpha \gamma}-c_{\alpha \gamma}$, where $c_{\alpha \gamma}$ is the coefficient of $x^\gamma$ in the polynomial $f_\alpha$ of $G$. Therefore, $\mathrm L_H$ is a closed subscheme of $\mathrm{St}_J^{\prec_n}$, hence a locally closed subscheme of $\mathcal{H}ilb^n_{\overline p(t)}$, by Theorem \ref{drv}.
\end{proof}

\begin{remark}\label{families}
Note that, similarly to the Hilbert scheme, the scheme $\mathrm{L}_H$ contains not only the $x_n$-liftings of $H$ in $\Kox$, which correspond to the closed $K$-points of $\mathrm{L}_H$, but also families of $x_n$-liftings, which correspond to $x_n$-liftings in $\Aox$, with $A$ a $K$-algebra and $\mathrm{Spec}(A)$ as space of parameters. 

More generally, let  us consider  a $K$-algebra  $R$ such that $U=\Spec(R)$ is  the space of parameters of a flat family of homogeneous   ideals  in $K[\mathbf x]$.  If $ \overline p(t)$ is the Hilbert polynomial of the  schemes in $\mathbb P^n$ defined by the extensions of such ideals  to $K[\mathbf x, x_n]$,   we can define the scheme of $x_n$-liftings $\mathrm{L}_U$ of $U$.  In a natural way $\mathrm{L}_U$  can be seen as   a locally closed subscheme of $\mathcal{H}ilb^n_{\overline p(t)}$. 
\end{remark}

\begin{corollary}\label{cor4}
For every homogeneous ideal $H\subset \Kx$, the locus of the radical liftings  of $H$ is an open subset in $\mathrm {L}_H$. 
\end{corollary}
\begin{proof}
By Proposition \ref{cor1}, ${\mathrm {L}}_H$ can be embedded in a Hilbert scheme $\mathcal{H}ilb^n_{\overline p(t)}$, hence   the locus of radical ideals of ${\mathrm {L}}_H$ is an open subscheme of ${\mathrm {L}}_H$, because the locus of points of $\mathcal{H}ilb^n_{\overline p(t)}$ corresponding to reduced schemes is open (see \cite[Th\'{e}or\`{e}me (12.2.1)(viii)]{Gro}). 
\end{proof}

{
It is well-known that  the Hilbert scheme is invariant under the action of the general linear group $GL_{n+1}$ of the invertible matrices corresponding to the change of coordinates in $\mathbb P^n$. The scheme of liftings  $\mathrm{L}_H$ is not invariant under the action of the whole $GL_{n+1}$, but it is  interesting to consider the action of some subgroups of $GL_{n+1}$ on it. 

For instance, the more general definition of $x_n$-liftings recalled in Remark \ref{piugenerale} can be reformulated saying that the action of any element $\theta$ of $GL_{n}$, trivially extended to $\mathbb P^n$ by setting $x_n \rightarrow x_n$, transforms $\mathrm{L}_H$ into the scheme of $x_n$-liftings of $\theta(H)$, that is $\theta(\mathrm L_H)=\mathrm{L}_{\theta(H)}$.  If, in particular, $H$ is fixed by the action of $\theta $, then $\theta(\mathrm L_H)=\mathrm L_H$.  We will use a similar group action in the proof of Theorem \ref{thm:liftACMdue}.

The scheme $\mathrm{L}_H$ is also invariant under the action of the  subgroup $\{\mu_t\,  :  t\in K^*\, \}$ of $GL_{n+1}$, where $\mu_t$ corresponds  to  $x_i \rightarrow x_i$ for every $i\not=n$ and $x_n\rightarrow tx_n$. This subgroup is canonically isomorphic to the torus $K^*$ and we will refer to this action as the {\it torus action} on $\mathrm{L}_H$.
 Note that $H\Kox$ is the unique fixed point of $\mathrm L_H$. 

In the following results we exploit the invariance of $\mathrm{L}_H$ under the torus action to prove that the ideal $\mathfrak h_0$ defining it is homogeneous with respect to a non-standard grading. An analogous result holds for Gr\"obner strata (see \cite{FR}); for similar more general results about varieties with a group action we refer to \cite{BB73,BB76,Fo}.

\begin{proposition} \label{prop:smooth}
The ideal $\mathfrak h_0\subset K[C]$ of Definition \ref{defh0} is homogeneous with respect to the grading induced by the weight vector $\omega(C_{\alpha \gamma})=\gamma_n+1$, where $C_{\alpha\gamma}$ is the coefficient of the term $x_nx^\gamma$ in \eqref{code} and $\gamma_n$ is the exponent of $x_n$ in $x^\gamma$. 
\end{proposition}

\begin{proof}
Let $\mathcal G$ be the set of polynomials in $K[\mathbf x, x_n, C]$ of \eqref{code} that we use to construct the ideal $\mathfrak h_0$. The polynomials of $\mathcal G$ are homogeneous with respect to the grading in $K[\mathbf x ,x_n,C]$ induced by the weights $\omega(x_i)=1$ for $i=0, \dots, n-1$, $\omega(x_n)=0$ and  $\omega(C_{\alpha \gamma})=\gamma_n+1$. Then, by construction, also 
$\mathfrak h_0$ is a homogeneous ideal with respect to this grading.
\end{proof}

\begin{corollary}\label{cor2}
Let $\overline y_0$ be  the point of the scheme of $x_n$-liftings  $\mathrm {L}_H$ that corresponds  to $HK[\mathbf x, x_n ]$. Then
\begin{enumerate}[(i)]
\item  all the irreducible components of  $\mathrm {L}_H$ contain the point $\overline y_0$;
\item  $\mathrm {L}_H$  can be isomorphically embedded into the Zariski tangent space at the point $\overline y_0$;
\item  if $\overline y_0$ is smooth, then $\mathrm{L}_H$ is isomorphic to an affine space.
\end{enumerate}
\end{corollary}

\begin{proof}
For  (i), we assume that $K$ is algebraically closed. In fact,  if $\Spec(\overline K[C]/\mathfrak h_0\overline K[C]) $ is connected, then also $\Spec(K[C]/\mathfrak h_0)$ is. 
Let $y$ be a closed point of $\mathrm {L}_H$ and let  $I\subseteq K[\mathbf x ,x_n]$ be the corresponding  $x_n$-lifting of $H$. 
If $G=\{f_\alpha\}_\alpha$ is the reduced Gr\"obner basis of $H$  and $G_I=\{f_\alpha+\sum c_{\alpha\gamma}x_nx^\gamma\}_\alpha$ is  the reduced Gr\"obner basis of $I$, for every $t\in K$ we let $I(t):=\mu_t I$  be the ideals obtained by the torus action on $I$. Then, 
$G(t):=\lbrace f_\alpha+\sum c_{\alpha\gamma}t^{\gamma_n+1}x_nx^\gamma\rbrace_\alpha$ 
is the Gr\"obner basis  of $I(t)$  for every $t\in K^*$, in particular $G(1)=G_I$ is that of $I$; moreover   for $t=0$ we have  $G(0)=G$. In this way we obtain a morphism of schemes $\phi_y \colon \mathbb A^1 \rightarrow \mathrm {L}_H$ such that  $\phi(1)=y$ and $\phi(0)=\overline y_0$.   

To prove (ii) and (iii), see for example \cite[Corollary 3.3]{FR}. 
\end{proof}  

\begin{remark}\label{eliminazione}  The above result has  interesting consequences for what concerns the efficiency of the computations.
Quoting the introduction of  \cite{CV}, when we perform the computation of ideals defining Gr\"obner stata \lq\lq many of the equations \dots contain parameters that appear in degree $1$ and that can be eliminated\rq\rq. This same phenomenon can be observed about the ideals that define schemes of liftings. Corollary \ref{cor2} (iii)  and the similar result  for Gr\"obner stata  \cite[Corollary 3.3]{FR} explain  the reason that produces this effect:   the elimination of variables through  elements of $\mathfrak h_0$  corresponds  to the embedding  of $\mathrm L_H$ into the Zariski tangent space at the point $\overline y_0$.  A very efficient way to produce this elimination consists in computing the Gr\"obner basis of $\mathfrak h_0\subset K[C]$ with respect to the  lexicographic  term order  with the variables $C$ ordered according to the weight vector $\omega(C_{\alpha \gamma})=\gamma_n+1$ of Proposition \ref{prop:smooth}. Indeed,  if  $h\in \mathfrak h_0$ is a homogeneous polynomial w.r.t. the weights given by $\omega$ and has a non-zero part of standard degree $1$, then  $\Ht(h)$  is a  term of standard degree $1$, namely is a variable $C_{\alpha \gamma}$,   and   $h$ allows the elimination of this variable.
\end{remark}
}


\section{Liftings of an arithmetically Cohen-Macaulay scheme of codimension $2$}
\label{sec:acm}

In this section, we study the $x_n$-liftings of a  saturated  ideal $H$ defining   an arithmetically Cohen-Macaulay (aCM for short) scheme  of codimension $2$ in $\mathbb P^{n-1}_K$, hence $n\geq 3$.  {\bf For sake of simplicity, in the following  we will call aCM ideal of codimension $2$ every homogeneous saturated ideal defining a  scheme with these properties.} 

Given such an  ideal $H\subset \Kx$ and a $x_n$-lifting $I\subset \Kox$ of $H$,  we denote by $G=\{f_\alpha\}_\alpha$   the reduced Gr\"obner basis of $H$ w.r.t.~the degrevlex order $\prec$ in $K[\mathbf x]$ and by  $G_I=\{f_\alpha+g_\alpha\}_\alpha$  the reduced Gr\"obner basis of $I$ w.r.t~$\prec_n$, as in Theorem \ref{FM}. Note that $\prec_n$ is the degrevlex order in $\Kox$.  We denote by $\mathfrak j$ the initial ideal $\mathrm{in}_{\prec}(H)$ and by $J$ the initial ideal $\mathrm{in}_{\prec_n}(I)$. 
Recall that these ideals are generated by the same minimal set of terms  $B_{\mathfrak j}$. Let moreover $p(t)$ be the Hilbert polynomial of $K[\mathbf x]/\mathfrak j$ and $\overline p(t)$ that of $\Kox/J$.

Up to a suitable linear change of coordinates, the variables $x_2,\ldots,x_{n-1}$ form a regular sequence for $\Kx/H$ and also for $\Kx/\mathfrak j$. Indeed, being $\Kx/H$  Cohen-Macaualy of codimension two, by the graded prime avoidance lemma, we can find a regular sequence for $\Kx/H$ consisting of $n-2$ linear forms $l_2,\ldots,l_{n-1}$. After the linear change of coordinates $\phi$ that sets $l_2\mapsto x_2,\ldots,l_{n-1}\mapsto x_{n-1}$, by \cite[Lemma (2.2)]{BS2} we obtain that the initial ideal of $\phi(H)$ is generated by terms which are not
divisible by $x_2,\ldots,x_{n-1}$ (see also Remark \ref{rem:degreverse}).
In particular, $\mathfrak j $ is a Cohen-Macaulay ring of codimension $2$. 

From now \bcr on \ecr, we  may suppose that $x_2,\ldots,x_{n-1}$ form a regular sequence for $\Kx/H$ and  for $\Kx/\mathfrak j$. Therefore, $B_{\mathfrak j}$ is contained in $ K[x_0, x_1]$.

{
\begin{lemma}\label{aCM}  In the above hypotheses, if $I\subset \Kox$ is a $x_n$-lifting of $H$, then $I$ defines an aCM scheme of codimension $2$ in $\mathbb P^n$.
\end{lemma} 

\begin{proof} 
By definition of $x_n$-lifting, $x_n$ is not a zero-divisor in $K[\mathbf x, x_n]/I$ and $K[\mathbf x, x_n]/(I+(x_n))$ is canonically isomorphic to $K[\mathbf x]/H$.  Hence, $x_n, x_{n-1}, \dots, x_{2}$ is a regular sequence for $K[\mathbf x, x_n]/I$. Moreover, $\mathrm{in}_{\prec_n}(I)=(B_{\mathfrak j})\Kox$,  hence the  Krull dimension of $K[\mathbf x, x_n]/I$ is $n-2$. 
\end{proof}
}
\begin{remark}
Note that the statement of Lemma \ref{aCM} holds also for a $x_n$-lifting of any saturated ideal defining an aCM scheme. 
\end{remark}

Recall that, being $\Kx/H$ Cohen-Macaulay of codimension two, there is a graded free resolution of type:
\begin{equation}\label{eq:acm resolution}
0 \longrightarrow \Kx^{a} \buildrel{\psi_2}\over\longrightarrow \Kx^{a+1} \buildrel{\psi_1}\over\longrightarrow \Kx \longrightarrow \Kx/H \longrightarrow 0,
\end{equation} 
and the Hilbert-Burch Theorem guarantees that the $a\times a$ minors of the $a\times (a+1)$ matrix of the homomorphism $\psi_2$ in \eqref{eq:acm resolution} form a set of generators for $H$. The rows of this matrix generate the first syzygies module of $H$. Vice versa, a $2$-codimensional scheme defined by the $a\times a$ minors of a matrix of type $a\times (a+1)$ of maximal rank is an aCM  scheme (e.g., see \cite[Theorem 20.15 and the paragraph that follows the proof]{Ei}).

For the sake of completeness, by the following Lemma we recall some standard results about syzygies (see for example \cite{KR1}) and a useful fact proved in \cite{FlR}. We will call a {\em lifting of a minimal free resolution of $\mathfrak j$} every complex obtained lifting the syzygies of $B_{\mathfrak j}$ in the usual sense of the theory of Gr\"obner bases. \bcR See \cite[Lemma 3.2]{Co} for a result similar to Lemma \ref{lemma:sizigie}\eqref{lemma:sizigie_ii}.\ecr

\begin{lemma}\label{lemma:sizigie}
In the above setting, 
\begin{enumerate}[(i)]
\item\label{lemma:sizigie_i} replacing $x_n$ by $0$ in a first syzygy $S=(s_\alpha)_\alpha$ of $G_I$, we get a first syzygy of $G$;
\item\label{lemma:sizigie_ii} there is a free resolution of $H$ (resp. $I$) of type \eqref{eq:acm resolution}, that is obtained by lifting a minimal free resolution of $\mathfrak j$ (resp. $J$).
\end{enumerate}
\end{lemma}

\begin{proof}
(i) Recall that the head terms of the polynomials of $G$ are the same as the head terms of the polynomials of $G_I$ (see Theorem \ref{FM}) and the polynomials $g_\alpha$ are all divisible by $x_n$. If, for every $\alpha$, we distinguish the part of $s_\alpha$ divisible by $x_n$ setting $s_\alpha=s'_\alpha+x_ns''_\alpha$, where $x_n$ does not occur in $s'_\alpha$, we have
$\sum_\alpha (s'_\alpha+x_ns''_\alpha)(f_\alpha+g_\alpha) = 0$, hence
$\sum_\alpha (s'_\alpha f_\alpha) + \sum_\alpha (x_n s''_\alpha f_\alpha) + \sum_\alpha (s'_\alpha+x_ns''_\alpha)g_\alpha = 0$,
and $x_n$ does not occur in the first sum, but the second and the third sums are divisible by $x_n$. Thus, we have $\sum_\alpha (s'_\alpha f_\alpha)=0$.

(ii) We show that a free resolution of type \eqref{eq:acm resolution} can be constructed starting from the reduced Gr\"obner basis of $H$. We can do the same 
for $I$, because also the ideal $I$ is the saturated defining ideal of a $2$-codimensional aCM closed subscheme (in $\mathbb P^n_A$, in this case),
by the definition of $x_n$-lifting. 

Being $\mathfrak j$ an  aCM ideal of codimension two, by the Hilbert-Burch Theorem we have the following minimal free resolution, which is constructed starting from $B_{\mathfrak j}$:
\begin{equation}\label{eq:jresolution}
0 \longrightarrow \Kx^{a} \buildrel{\psi'_2}\over\longrightarrow \Kx^{a+1} \buildrel{\psi'_1}\over\longrightarrow \Kx \longrightarrow \Kx/\mathfrak j 
\longrightarrow 0
\end{equation} 
where the rows of the matrix $M_{\mathfrak j}$ of $\psi'_2$ form a minimal set of generators of the first syzygies of $B_{\mathfrak j}$. 
By the properties of Gr\"obner bases, we obtain a set of generators of the first syzygies of $G$ by lifting the rows of $M_{\mathfrak j}$ and thus 
a new matrix $M_H$. By $M_H$ we can construct a complex of type \eqref{eq:acm resolution} that is also a free resolution because it is exact in $\Kx^{a+1}$ by construction, and therefore the rows of $M_H$ are independent.
\end{proof}

\begin{remark}\label{rem:matrici}
Consider the matrix $M$ obtained by the matrix $M_H$ of $\psi_2$ in \eqref{eq:acm resolution} adding to each entry of $M_H$ a linear combination of the terms of the appropriate degree that are divisible by $x_n$. By the Hilbert-Burch Theorem   the $a\times a$ minors of the matrix $M$ generate an  aCM ideal of codimension $2$.  Furthermore, $I$ is a $x_n$-lifting of $H$. Thus, by Lemma \ref{lemma:sizigie}, all the $x_n$-liftings of $H$ are of this type. 
\end{remark}

The points of a Hilbert scheme corresponding to aCM schemes form an open subset   \cite[Th\'{e}or\`{e}me (12.2.1)(vii)]{Gro}, which is also smooth if the aCM schemes are of codimension two \cite[for $n=2$]{Fo},\cite[Theorem 2(i), for $n\geq 3$]{Elli}.   Hence, there is only one irreducible component $Y$ of $\mathcal{H}ilb^{n-1}_{p(t)}$ containing the point $y_0$ corresponding to $H$ and only one irreducible component   $\overline Y$   of $\mathcal{H}ilb^n_{\overline p(t)}$ containing the point $\overline y_0$ corresponding to $\overline H:=H\Kox$. 
Observe that all the $x_n$-liftings of $H$ belong to $\overline Y$, by Corollary \ref{cor2}(i).

\begin{theorem}\label{th:spazio affine}
Let $H\subset \Kx$ be the saturated homogeneous ideal defining an aCM scheme of codimension two. 
The scheme $\mathrm L_H$ is isomorphic to an affine space on the infinite field $K$.
\end{theorem}

\begin{proof}
We show that $\mathrm L_H$ is smooth at  $\overline y_0$ and then  apply Corollary \ref{cor2} (iii). To this aim we determine  the dimension of the Zariski tangent space $Z_{\overline y_0}(\mathrm L_H)$  to $\mathrm L_H$ at   $\overline y_0$. This dimension does not change if we replace  $K$ by its algebraic closure. Therefore we may assume that $K$ is algebraically closed.

Let $\mathcal U$ be the set of points in $Y$ corresponding to the ideals $\mathfrak d\subset K[\mathbf x]$ such that $\cN(\mathfrak j)$ is a basis of the $K$-vector space $K[\mathbf x]/\mathfrak d$. We can see that $\mathcal U$ is a locally closed subscheme of $Y$. Indeed, it is containd as an open subset in the locus of points of $Y$ with the same Hilbert function as $\Kx/\mathfrak j$, which  is locally closed in $Y$ (see \cite{Mall00}). Furthermore, $\mathcal U$ is an affine scheme $\mathcal U=\mathrm{Spec}(R)$, where $R=K[E]/\mathfrak p$ is a finitely generated $K$-algebra (see \cite{CMR}: observe that $B_{\mathfrak j}$ generates a quasi-stable ideal). Note that the point $y_0$ belongs to $\mathcal U$: we can assume that it is given by  the maximal ideal $(E)R$, hence  $\mathfrak p\subset (E)$.

In a similar way we define   the locally closed subscheme  $\mathcal V  $   of  $\overline Y$ whose closed points correspond to the ideals $D\subset K[\mathbf x, x_n]$ such that   $\cN(J)$ is a basis of the $K$-vector space $K[\mathbf x,x_n]/D$.  
Applying \cite[Proposition 1.1]{FlR} to $J$ (note that $n\geq 3$), we see that    $\mathcal V$ is in fact open in $\overline Y$. 
Moreover,  we observe that for every ideal $D$ corresponding to a point of $\mathcal  V$, $x_n$ is not a zero-divisor on $\Kox/D$. Therefore, the image of $D$ by the canonical morphism $ \phi \colon \Kox \rightarrow  K[\mathbf x]$, which transforms $x_n$ in $0$, defines a point of $\mathcal U$.  Therefore, $\phi$ determines a  projection $\pi \colon \mathcal  V \rightarrow \mathcal  U$.

On the other hand, if $\mathfrak d$ is a point of $\mathcal U$, then $D=\mathfrak d \otimes_K K[x_n]$ defines a point of  $\mathcal V$, hence a section  $\sigma \colon \mathcal U \rightarrow \mathcal  V$ of $\pi$, being $\pi \circ \sigma=Id_U$. In particular, $\sigma(y_0)=\overline y_0$.

More generally,   all the $x_n$-liftings of $\mathfrak d$ belong to $\mathcal V$. Hence, $\mathcal V$  is  the scheme $\mathrm L _{\mathcal U}$ of  the $x_n$-liftings of  $\mathcal U$ (see Remark \ref{families}). 


By construction, for every closed point $y$ of $\mathcal  U$, the fibre $\pi^{-1}(y)$ is the scheme of the $x_n$-liftings of $y$. We are interested in the fibre $\pi^{-1}(y_0)$, which is $\mathrm L_H$.
If $m:=dim(\mathcal U)$ and $\overline m:=dim(\mathcal V)$, by general results about the dimension of the fibres of a morphism (e.g.~\cite[Chapter II, Ex. 3.22]{H77}), the dimension of $\mathrm L_H$ is $\geq \overline m-m$, thus the dimension of the Zariski tangent space $Z_{\overline y_0}(\mathrm L_H)$ is $\geq \dim(\mathrm L_H)\geq \overline m-m$.
Now, we prove that $\mathrm L_H$ and  $Z_{\overline y_0}(\mathrm L _H)$ have the same dimension $\overline m -m$, and then $\mathrm L_H$ is smooth at $\overline y_0$.

By applying the construction of the $x_n$-liftings to $\mathcal  U$ we described in Sections \ref{sec:lifting functor} and \ref{sec:representability}, we obtain $\mathcal  V$ as an affine scheme $\mathcal  V=\mathrm{Spec}(S)$ where $S=R[C]/\mathfrak h_{\mathcal U}=K[E,C]/(\mathfrak p \cup \mathfrak h_{\mathcal  U})$, and $\mathfrak h_{\mathcal  U}$ is the analogous of the ideal of Definition \ref{defh0}.  By construction, $\mathfrak h_{\mathcal  U}$ is contained in  $(C)K[E,C]$. Note that, specializing the variables $E$ to $0$ in $\mathfrak h_U$, we obtain precisely the ideal $\mathfrak h_0 \subset K[C]$ defining the scheme $\mathrm L_H$. Therefore, the linear part $l$ of every element in $\mathfrak h_U$  coincides with the linear part of the corresponding element of  $\mathfrak h_0 $, since $l\in K[C]$.

%


The Zariski tangent space $Z_{\overline y_0}(\mathcal  V)$ to $\mathcal  V$ at $\overline y_0$ has dimension $\overline m$, because $\overline y_0$ is a smooth point of $\overline Y$ and   $\mathcal  V$ is open in $\overline Y$. It is defined by the vanishing of the elements in the $K$-vector space $W$ of the linear parts of the elements in $\mathfrak p \cup \mathfrak h_{\mathcal  U}$. Observe that $W=W_1+W_2$, where $W_1$ contains the linear parts of elements of $\mathfrak p$ and $W_2$ contains the linear parts of elements of $\mathfrak h_{\mathcal U}$, and the sum is direct, being $W_1 \subset K[E]_1$ and $W_2 \subset K[C]_1$.

The vector space $W_2$ is also given by the linear parts of elements in $\mathfrak h_0$, hence its vanishing also defines the Zariski tangent space $Z_{\overline  y_0}(\mathrm L_H)$, while the vanishing of the vector space $W_1$ also defines the Zariski tangent space $Z_{ y_0}(\mathcal U)$.

Therefore, $\dim(Z_{\overline  y_0}(\mathrm L_H))=\vert C\vert- dim(W_2)=\vert C\vert +\vert E\vert- dim(W )-(\vert E\vert- dim(W_1)) =\dim(Z_{\overline  y_0}(\mathcal V))-\dim(Z_{  y_0}(\mathcal U))\leq \overline m -m$.


As a consequence, we have $ \overline m -m\geq  \dim(Z_{\overline y_0}(\mathrm L_H)) \geq \dim(\mathrm L_H) \geq \overline m -m$,  hence $\dim(\mathrm L_H) =\dim(Z_{\overline y_0}(\mathrm L_H))=\overline m-m$, so that  $\mathrm{L}_H$ is smooth at $\overline y_0$. By Corollary  \ref{cor2}, $\mathrm L_H$ is isomorphic to an affine space of dimension $\overline m -m$ over $K$.
\end{proof}

\begin{remark}
A direct consequence of the proof of Theorem \ref{th:spazio affine} is that $\mathcal U$  is smooth at $y_0$, being $y_0$ any of its points. We observe that all the points of $\mathcal U$ share the same Hilbert function, and $\mathcal U$ is an open subset of a Hilbert function stratum of $Y$.    In particular, if $n=2$,   $Y$ is the only component of the Hilbert scheme and all  its points correspond to aCM schemes of codimension two in $\mathbb P^2$.   Hence, for an infinite field of any characteristic, we find  a result analogous to that described in \cite{Go88} for a field of  characteristic 0 about the smoothness of the  Hilbert function strata. 
\end{remark}

{
\begin{remark}
Let $\mathfrak j\subset K[x_0,x_1,x_2]$ be any aCM ideal of codimension $2$     with $B_\mathfrak j\subset K[x_0,x_1]$. 
The construction of the Gr\"obner stratum  $\mathrm{St}_\mathfrak j^{\prec_{lex}}$ w.r.t. the  lex term order $\prec_{lex}$  has been studied in \cite{CV}, proving that $\mathrm{St}_{\mathfrak j}^{\prec_{lex}}$ is an affine space and providing an explicit parameterisation of the matrices of the first syzygies of their Pommaret basis. 

In \cite{Co}  similar results are obtained for the Gr\"obner stratum  $\mathrm{St}_{\mathfrak j}^{\prec}$ w.r.t. the graded  reverse lexicographic  term order $\prec$ assuming that $\mathfrak j\subset K[x_0,x_1]$ is a lex-segment ideal.

In several respects, the framework and the methodologies  of this section  recall those  in \cite{CV,Co}. In  fact,  the objects of study are families of ideals whose initial ideals  are generated by $B_{\mathfrak j}$,  it is proved that  such families can be parameterised by    affine spaces and in the proofs  Hilbert-Burch resolutions and Pommaret bases are used (note that we actually  use Pommaret bases in the proof of Theorem \ref{th:spazio affine}, when we refer to \cite{CMR}).

On the other hand, there are  no common  results among  the present paper and the two quoted ones. The main reason is that the families of ideals considered in \cite{CV,Co} define subschemes in  $\mathbb P^2$, while the  liftings of  an aCM ideal $H$  of codimension $2$ corresponds  to subschemes in $\mathbb P^n$ with $n\geq 3$. 

It would be very interesting to find an explicit  parameterisation of $\mathrm L_H$  by Hilbert-Burch resolutions, similar to those  of \cite{CV,Co}.  Anyway, the results of  \cite{CV,Co}   cannot be easily generalized.  Indeed, we cannot generalize the resolution of  \cite{CV},  since not all the liftings  have initial ideal  w.r.t. $\prec_{lex}$ generated by $B_{\mathfrak j}$.  Let us consider, for instance,  the aCM ideals  $\mathfrak j=H=(x_0^2,x_0x_1,x_1^2)\subset K[x_0, x_1, x_2]$,  $J=\mathfrak j K[x_0, x_1, x_2,x_3]$ and $I=(x_0^2-x_1x_3,x_0x_1+x_3^2,x_1^2+x_0x_3)\subset K[x_0,x_1,x_2,x_3]$. The ideal $I$ is a  $x_3$-lifting of $H$, but does not belong to $ \mathrm{St}_J^{\prec_{lex}}$. Indeed the initial ideal of $I$ w.r.t.~lex is the  ideal $J'=(x_0^2,x_0x_1,x_0x_2,x_1^3)$, which  is not an aCM ideal. 
Note that the following matrix of syzygies of $I$ 
\[
\begin{pmatrix}
x_1 & -x_0 &  x_3\\
x_3 & x_1 & -x_0
\end{pmatrix}
\] (and also all the possible others) has no null entries, hence  an analogous of \cite[Lemma 3.5]{CV} is not achievable.
Indeed, the embedding of   $\mathrm L_H$   in a Gr\"obner stratum $\mathrm{St}_J^{\sigma}$ is possible only if  $\sigma$ is a degreverse  term order.  For  this reason  it could be more suitable to consider the framework of \cite{Co}, but there the main result  is proved under the limiting condition that $B_\mathfrak j$  is a lex-segment in $K[x_0,x_1]$. 
\end{remark}
}

\begin{theorem}\label{thm:liftACMdue}
Every saturated homogeneous ideal $H$ defining a $2$-codimensional aCM scheme over an infinite field $K$ has a radical $x_n$-lifting. In particular, every saturated homogeneous ideal defining a $0$-dimensional scheme in $\mathbb P^2$ has a radical $x_n$-lifting in $\mathbb P^3$. 
\end{theorem}

\begin{proof}
For every $x^\alpha\in B_{\mathfrak j}$, let $f_\alpha:=x^\alpha + \sum_\gamma c_{\alpha,\gamma} x^\gamma$ be the polynomial of the reduced 
Gr\"obner basis $G_H$ with head term $x^\alpha$. Then, let $\omega=[\omega_0,\ldots,\omega_{n-1}]$ be a weight vector such that the polynomials
$f_\alpha(t):=x^\alpha + \sum_\gamma c_{\alpha,\gamma} t^{\omega\cdot(\alpha-\gamma)}x^\gamma$ generate a flat family $\{H(t)\}_t$ of Gr\"obner deformations from $H$ to $\mathfrak j$ (e.g.~\cite{BM91}).   Recall that the ideal of this family corresponding to each  $\bar t\in K^*$ 
is isomorphic to $H$ since $\phi_t(H)=H(t)$ by  the automorphism $\phi_t$ of $\mathbb P^{n-1}$ given by $x_0\mapsto t^{-\omega_0}x_0, \dots, x_{n-1}\mapsto t^{-\omega_{n-1}}x_{n-1}$. Moreover, $H(0)=H$, being all the exponents $\omega\cdot(\alpha-\gamma)$ positive.

For $H(t)$ there is a resolution of type $\eqref{eq:acm resolution}$ we obtain as in Lemma \ref{lemma:sizigie}\eqref{lemma:sizigie_ii}. Let $M_{H(t)}$ be the matrix of the corresponding homomorphism $\psi_2$. Observe that each addend of the entries of $M_{H(t)}$ is divisible by $t$, except for the addends that form the syzygies of $\mathfrak j$, according to Lemma \ref{lemma:sizigie}\eqref{lemma:sizigie_ii}. Let $N$ be a radical lifting of $\mathfrak j$, which exists by \cite{H66} or \cite[Th. 2.2]{GGR} or \cite[Th. 8]{Ro} (see also \cite[Theorem 6.2.12]{KR2}), and let $M_N$ be the matrix corresponding to the homomorphism $\psi_2$ in a free resolution of $N$ of type \eqref{eq:acm resolution}. Each addend of the entries of $M_{N}$ is divisible by $x_n$, except for the addends appearing in the syzygies of $\mathfrak j$, according to Lemma \ref{lemma:sizigie}\eqref{lemma:sizigie_ii}.
Let $M_{\mathfrak j}$ be the analogous matrix for $\mathfrak j$.

Now, consider the matrix $M(t):=M_{H(t)}+M_N-M_{\mathfrak j}$ and let $\{I(t)\}_t$ be the family of the ideals generated by the maximal minors
of $M(t)$. Hence, $M(t)$ is the matrix of the homomorphism $\psi_2(t)$ of a resolution
$$0\longrightarrow (\Kx[t])^{a} \buildrel{\psi_2(t)}\over\longrightarrow (\Kx[t])^{a+1} \buildrel{\psi_1(t)}\over\longrightarrow \Kx[t] \longrightarrow \Kx[t]/I(t)\longrightarrow 0.$$
By construction (see also Remark \ref{rem:matrici}), $\{I(t)\}_t$ is a flat family of $x_n$-liftings of $\{H(t)\}_t$  parameterised by an affine line $\mathbb A^1$. Hence, it is embedded in a Hilbert scheme. 

By the already cited \cite[Th\'{e}or\`{e}me (12.2.1), (viii)]{Gro}, the locus of points of a Hilbert scheme corresponding to reduced schemes is open. Hence, there is an
open subset $U$ of $\mathbb A^1$ corresponding to reduced schemes. Moreover, $U$ is not empty because it contains $t=0$. Due to the fact that
the field $K$ is infinite, there is at least another value $\bar t\not=0$ in $U$, so that $I(\bar t)$ is reduced. Then, $H$ has the radical
$x_n$-lifting obtained from  $I(\bar t)$ by the automorphism $\phi_{\overline t}^{-1}$ of $\mathbb P^n$ that is obtained extending $\phi^{-1}_{\overline t}$ putting $\phi^{-1}_{\overline t}(x_n)=x_n$.
\end{proof}


\section{Examples}
\bcR The computations about the following examples have been performed using the softwares CoCoA  \cite{CoCoA-5} and Maple16 \cite{Maple}. The Gr\"obner bases of  the ideals $\mathfrak h_0$ in $K[C]$ that define the schemes of $x_n$-liftings are those  w.r.t. the lexicographic term order  with the variables $C$ ordered according to the weights defined in  Proposition \ref{prop:smooth}. This term order naturally leads to the complete  elimination of eliminable variables  
(see Remark \ref{eliminazione}).\ecr


\begin{example}\label{ex:LuoYilmaz}
We study the $x_3$-liftings of the lex-segment ideal  
\[
J:=(x_0^2, x_0x_1, x_0x_2, x_1^2)\subseteq K[x_0,x_1,x_2],
\]
finding a scheme having two distinct irreducible components, one of which is made of non-radical liftings.
This is the same example as \cite[Example 3.3]{LY}, after a change of coordinates allowing us to consider a monomial ideal. 
In \cite{LY}, the authors explicitly compute a set of equations defining $\mathrm L_J$ with a different technique: they impose conditions on the syzygies
of a set of polynomials generating a $x_3$-lifting of $J$. In general, the algorithm they use leads to a different set of conditions from the ones 
we compute by Theorem \ref{FM} and Proposition \ref{comeLR}. However, on this example, the equations defining the scheme 
$\mathrm L_J$ obtained in \cite{LY} are the same as those that we obtain with our strategy. Here we just briefly expose the equations and we 
focus on the structure of $\mathrm L_J$. 

Starting from the set $\mathcal G=\lbrace f_1+g_1,f_2+g_2,f_3+g_3,f_4+g_4\rbrace$ with
\[
\begin{array}{lcl}
f_1=x_0^2,& \quad & g_1=C_1x_0x_3+C_2x_1x_3+C_3x_2x_3+C_4x_3^2,\\
f_2=x_0x_1,& \quad & g_2=C_5x_0x_3+C_6x_1x_3+C_7x_2x_3+C_8x_3^2,\\
f_3=x_0x_2,& \quad & g_3=C_9x_0x_3+C_{10}x_1x_3+C_{11}x_2x_3+C_{12}x_3^2,\\
f_4=x_1^2,& \quad & g_4=C_{13}x_0x_3+C_{14}x_1x_3+C_{15}x_2x_3+C_{16}x_3^2,\\
\end{array}
\]
we impose that $\mathcal G$ is a Gr\"obner basis in $K[x_0,x_1,x_2,x_3]$ with initial ideal $J$ w.r.t.~the degreverse term order that coincides with the deglex term order on 
$K[x_0,x_1,x_2]$. In this way, we obtain the ideal 
$\mathfrak h_0=(C_{{6}}-C_{{11}},C_{{2}},C_{{7}},C_{{3}},-C_{{10}}C_{{5}}-C_{{11}}C_{
{9}}-C_{{12}},-C_{{14}}C_{{5}}+{C_{{5}}}^{2}$ 
$+C_{{13}}C_{{11}}-C_{{13}}C_{{1}}-C_{{15}}C_{{9}}-C_{{16}},-C_{{5}}C_{{11}}-C_{{8}},-{C_{{11}}}^
{2}+C_{{1}}C_{{11}}+C_{{4}}, -C_{{10}}C_{{13}},-2\,C_{{11}}C_{{10}}+$ $C_{{1}}C_{{10}},2\,C_{{10}}C_{5}-C_{10}C_{14},-C_{10}C_{15})$.

After eliminating the 8 variables $ C_{2},C_{3},C_{4},C_{6},C_{7},$ $C_{8},C_{12},C_{16}$, we obtain an embedding of the scheme of liftings in $\mathbb A^8$ 
given by the ideal 
\begin{equation}\label{condFatt}
(C_{{10}})\cap (C_{{13}},C_{{15}}, C_{{1}}-2\,C_{{11}},C_{{14}}-2\,C_{{5}})
\end{equation}
Hence, $\mathrm L_J$ has two irreducible components:  $\mathrm L_1$, the hyperplane in $\mathbb A^8$  given  by  the ideal $(C_{10})$, 
and $\mathrm L_2$, the linear subspace of dimension $4$ in $\mathbb A^8$
given by the ideal $(C_{{13}},C_{{15}}, C_{{1}}-2\,C_{{11}},C_{{14}}-2\,C_{{5}})$. 
We will now explicitly compute the  locus of radical liftings of $\mathrm L_J$, which by Corollary \ref{cor4} is an open subset of $\mathrm L_J$, hence we compute the locus of radical liftings of both the components $\mathrm L_1$ and $\mathrm L_2$.

The first component $\mathrm L_1$  
parameterises the liftings that are generated by polynomials of
the following type:
\[
\begin{array}{cl}
f_1+g_1=&\left( x_0+C_{{11}}x_3 \right) \left( x_0+(C_{{1}}-C_{{11}})x_3 \right),\\
f_2+g_2=&\left( x_1+C_{{5}}x_3 \right)  \left( x_0+C_{{11}}x_3 \right), \\
f_3+g_3=&{x_1}^{2}+C_{{13}}x_0x_3+C_{{14}}x_1x_3+C_{{15}}x_2x_3+(C_{{14}}C_{{5}}-C_{{13}}C_{{11}}-{C_{{5}}}^{2}+C_{{13}}C_{{1}}+\\
        &+C_{{15}}C_{{9}})x_3^2,\\
f_4+g_4=&\left( x_2+C_{{9}}x_3 \right)  \left( x_0+C_{{11}}x_3 \right).
\end{array}
\]
By easy computations we can see that we obtain non reduced ideals only if either
$C_{{1}}-2C_{{11}}=0$ \ or \ $(C_{14}-2C_5)^2+4C_{13}(C_1-2C_{11})=C_{15}=0$.

The second irreducible component $\mathrm L_2$ of $\mathrm L_J$ parameterises the liftings generated by polynomials of the following type:
\[
\begin{array}{rcl}
f_1+g_1&=&\left( x_0+C_{{11}}x_3 \right) ^{2},\\
f_2+g_2&=& \left( 2\,x_1+C_{{14}}x_3 \right)  \left( x_0+C_{{11}}x_3 \right) , \\
f_3+g_3&=&\left( 2\,x_1+C_{{14}}x_3 \right) ^{2},\\
f_4+g_4&=&2\,x_0x_2+2\,C_{{9}}x_0x_3+2\,C_{{10}}x_1x_3+2\,C_{{11}}x_2x_3+(2\,C_{{11}}C_{{9}}+C_{{10}}C_{{14}})x_3^2.
\end{array}
\]
Each ideal of this type corresponds to a double structure over the line $x_0+C_{{11}}x_3=2\,x_1+C_{{14}}x_3=0$, hence the locus of radical liftings of $\mathrm L_2$ is empty. 
\end{example}


\begin{example}\label{esiso}
In this example, we apply the construction arising from Theorem \ref{FM} and Proposition \ref{comeLR} with the degrevlex and the deglex term orders to compute two different $K$-algebras that define the scheme $\mathrm L_H$ of an ideal $H$, thanks to Theorem \ref{funtrappr1}. Then, we find an explicit isomorphism between them.

Take the homogeneous ideal $H=(x_0^2,x_0x_1,x_1^4+x_1x_2^3)\subset K[x_0,x_1,x_2]$, with $x_0>x_1>x_2$.  
The reduced Gr\"obner basis of $H$ w.r.t.~the degrevlex term order is $G=\{f_1=x_0^2,f_2=x_0x_1, f_3=x_1^4+x_0x_2^3\}$. The reduced Gr\"obner basis of 
$H$ w.r.t.~the deglex term order is $G'=\{f_1'=x_0^2,f_2'=x_0x_1,f_3'=x_0x_2^3+x_1^4,f_4'=x_1^5\}$. 
We get $\mathcal G:=\{f_1+g_1,f_2+g_2,f_3+g_3\}$ with
$$\begin{array}{ll}
g_1=&C_{{1}}x_{{0}}x_{{3}}+C_{{2}}x_{{1}}x_{{3}}+C_{{3
}}x_{{2}}x_{{3}}+C_{{4}}x_{{3}}^{2},\\
g_2=&C_{{5}}x_{{0}}x_{{3}}+C_{{6}}x_{{1}}x_{{3}}+C_{{
7}}x_{{2}}x_{{3}}+C_{{8}}x_{{3}}^{2}, \\ 
g_3=&C_{{9}}x_{{1}}^{3}x_{{3}}+C_{{10}}x_{{1}}^{2}
x_{{2}}x_{{3}}+C_{{11}}x_{{0}}x_{{2}}^{2}x_{{3}}+C_{{12}}x_{{1}}x_{
{2}}^{2}x_{{3}}+C_{{13}}x_{{2}}^{3}x_{{3}}+C_{{14}}x_{{1}}^
{2}x_{{3}}^{2}+\\
&C_{{15}}x_{{0}}x_{{2}}x_{{3}}^{2}+C_{{16}}x_{{1}}x_
{{2}}x_{{3}}^{2}+C_{{17}}x_{{2}}^{2}x_{{3}}^{2}+C_{{18}}x_{{0}}x_{{3}}^{3}+
C_{{19}}x_{{1}}x_{{3}}^{3}+C_{{20}}x_{{2}}x_{{3}}^{3}+C_{{21}}x_{{3}}^{4
}
,\end{array}$$
Observe that, starting from $\mathcal G$, we construct the scheme $\mathrm L_H$ in an affine space of dimension 21. We will soon see that the scheme $\mathrm L_H$ is smaller than this ambient space. The set $\mathcal G$ is a Gr\"obner basis w.r.t.~the degrevlex term order, modulo the following ideal in $K[C]$:

$\mathfrak h_0=(C_{{2}},C_{{7}},C_{{13}}-C_{{1}}+C_{{6}},-C_{{6}}C_{{5}}+C_
{{8}},-C_{{6}}^{2}+C_{{1}}C_{{6}}-C_{{4}},
C_{{17}}-C_{{11}}C_{{1}}-C_{{12}}C_{{5}}+C_{{6}}C_{{11}}, 
C_{{6}}C_{{15}}+C_{{10}}C_{{5}}^{2}-C_{{16}}C_{{5}}+C_{{20}}-C_{{15}}C_{{1}},
-C_{{9}}{C_{{5}}}^{3}-C_{{19}}
C_{{5}}+
C_{{21}}+{C_{{5}}}^{4}+C_{{6}}C_{{18}}-C_{{18}}C_{{1}}+C_{{14}
}{C_{{5}}}^{2})$.

After the  elimination of the variables, we  obtain that  $\mathrm L_H$ is isomorphic to an affine space of dimension 12. 
Using this isomorphism, the polynomials of $\mathcal G$ become
$$\begin{array}{ll}
f_1+g_1=&\left( x_{{0}}+C_{{6}}x_{{3}} \right)  \left( C_{{5}}x_{{3}}+x_{{1}}
 \right)
,\\
f_2+g_2=&\left( x_{{0}}+C_{{6}}x_{{3}} \right)  \left( -C_{{6}}x_{{3}}+C_{{1}}
x_{{3}}+x_{{0}} \right) 
,\\
f_3+g_3=&x_{{1}}^{4}+x_{{0}}{x_{{2}}}^{3}+C_{{10}}x_{{1}}^{2}x_{{2}}x_{{3}}+ C_{{11}}x_{{0}}x_{{2}}^{2}x_{{3}}+C_{{12}}x_{{1}}x_{{2}}^{2}x_{{3}}+ 
C_{{15}}x_{{0}}x_{{2}}x_{{3}}^{2}+\\ 
&C_{{16}}x_{{1}}x_{{2}}x_{{3}}^{2}+C_{{9}}x_{{1}}^{3}x_{{3}}+C_{{14}}x_{{1}}^{2}x_{{3}}^{2}+
C_{{18}}x_{{0}}x_{{3}}^{3}+C_{{19}}x_{{1}}x_{{3}}^{3}+\left(C_{1} -C_{6} \right)
x_{{2}}^{3} x_{{3}}+\\
&\left( C_{{15}}C_{{1}}+C_{{16}}C_{{5}}-C_{{10}}{C_{{5}}}^{2}-C_{{6}}C_{{15}}
 \right) x_{{2}}x_{{3}}^{3}+ \\
&\left( C_{{12}}C_{{5}}+C_{{11}}C_{{1}}-C_{{6}}C_{{11}} \right)x_{{2}}^{2} x_{{3}}^{2}+\\
&\left( C_{{9}}C_{{5}}^{3}+C_{{19}}C_{{5}}-C_{{6}}C_{18}+C_{{18}}C_{{1}}-C_{{14}}C_{{5}}^{2}-C_{{5}}^{4} \right) x_{{
3}}^{4}.\end{array}$$
The polynomials of the reduced Gr\"obner basis of $(\mathcal G)$ modulo $\mathfrak h_0$ w.r.t.~to the deglex term order (in terms of the variables $C$) are: 
$$\begin{array}{ll}
p'_1=&{x_{{0}}}^{2}+C_{{1}}x_{{0}}x_{{3}}+ \left( C_{{1}}C_{{6}}-C_{{6}}^{
2} \right) {x_{{3}}}^{2},\\
p'_2=& x_{{0}}x_{{1}}+C_{{5}}x_{{0}}x_{{3}}+C_{{6}}x_{{1}}x_{{3}}+C_{{6}}C_{{5}}x_{{3}}^{
2},\\
p'_3=&x_{{0}}{x_{{2}}}^{3}+{x_{{1}}}
^{4}+C_{{10}}x_{{1}}^{2}x_{{2}}x_{{3}}+C_{{11}}x_{{0}}x_{{2}}^{2
}x_{{3}}+C_{{12}}x_{{1}}x_{{2}}^{2}x_{{3}}+C_{{15}}x_{{0}}x_{{
2}}x_{{3}}^{2}+\\
&C_{{16}}x_{{1}}x_{{2}}x_{{3}}^{2}+C_{{9}}x_{{1}}^{3}x_{3}+C_{14}x_{1}^{2}x_{3}^{2}+C_{18}x_{0}x_{{3}}^{3}+C_{{19}}x_{{1}}x_{{3}}^{3}+ \\
&\left( C_{{15}}C_{{1}}+C_{{16}}C_{{5}}-C_{{10}}C_{{5}}^{2}-C_{{6}}C_{{15}}
 \right) x_{{2}}x_{{3}}^{3}+\\
&+ \left( -C_{{6}}+C_{{1}} \right) 
x_{{2}}^{3}x_{{3}}+ \left( C_{{12}}C_{{5}}+C_{{11}}C_{{1}}-C_{{6}}C_{{11}}
 \right) x_{{2}}^{2}x_{{3}}^{2}+\\
&+ \left( C_{{9}}C_{{5}}^{3}+C_{{19}}C_{{5}}-C_{{6}}C_{{
18}}+C_{{18}}C_{{1}}-C_{{14}}C_{{5}}^{2}-C_{{5}}^{4} \right) x_{{
3}}^{4}\\
p'_4=&x_{{1}}^{5}+ \left( C_{{5}}C_{{11}}C_{{1}}+C_{12}{C_{{5}}}^{2}-2\,C_{{11}}C_{{6}}C_{{5}} \right)x_{{2}}^{2}x_{{3}}^{3}+\\
&\left( C_{{5}}C_{{15}}C_{{1}}-2\,C_{{15}}C_{{6}}C_{{5}}+C_{{16}
}C_{{5}}^{2}-C_{{10}}C_{{5}}^{3} \right) x_{{2}}x_{{3}}^{4}+ \\
&+\left( C_{{18}}C_{{1}}+C_{{9}}C_{{5}}^{3}-C_{{5}}^{4}+2\,C_{{19}}
C_{{5}}-2\,C_{{6}}C_{{18}}-C_{{14}}C_{{5}}^{2} \right) x_{{1}}x_{{3
}}^{4}+\\
& \left( C_{{5}}+C_{{9}} \right) x_{{1}}^{4}x_{{3}}+ \left( C_{{5}}C_{{9}}+C_{{14}} \right) x_{{1}}^{3}x_{{3}}^{2}+ \left( C_{5}C_{{14}}+C_{{19}} \right)x_{{1}}^{2}x_{{3}}^{3}+ \\
&\left( C_{5}C_{1}-2\,C_{{6}}C_{{5}} \right) x_{{2}}^{3}x_{{3}}^{2}+\left(2\,C_{{12}}C_{{5}}+C_{{11}}C_{{1}}-2\,C_{{6}}C_{{11}} \right) x_{{2}}^{2}x_{{1}}x_{{3}}^{2}+\\
& \left( C_{5}C_{10}+C_{16} \right) x_{{2}}x_{{1}}^{2}x_{{3}}^{2}+\left( 2\,C_{{16}}C_{{5}}-2\,C_{{6}}C_{{15}}+C_{{15}}C_{{1}}-C_{{10}}C_{{5}}^{2} \right) x_{{2}}x_{{1}}x_{{3}}^{3}+\\
&C_{{12}}x_{{1}}^{2}x_{{2}}^{2}x_{{3}}+C_{{10}}x_{{1}}^{3}x_{{2}}x_{{3}}+\left( -2\,C_{{6}}+C_{{1}} \right) x_{{2}}^{3}x_{{1}}x_{{3}}+\\
& \left( C_{{9}}C_{{5}}^{4}-C_{{14}}C_{{5}}^{3}-2\,C_{{5}}C_{{6}}C_{{18}}+C_{{19}}C_{{5}}^{2}+C_{{5}}C_{{18}}C_{{1}}-C_{{5}}^{5} \right)x_{{3}}^{5}.\end{array}$$
Starting from the Gr\"obner basis $G'$, consider $\mathcal G'=\{f'_1+g'_1,f'_2+g'_2,f'_3+g'_3,f'_4+g'_4\}$, where $g'_i\in K[D][\mathbf x,x_n]$ as in formula \eqref{code} with the varables $D$ in place of $C$. We highlight that $\vert D\vert =39\not=\vert C\vert$.

Now, we are going to construct $\mathrm L_H$ into an affine space of dimension 39. The set 
$\mathcal G'$ is a Gr\"obner basis w.r.t.~the deglex term order, modulo the ideal 
$\mathfrak h'_0= (D_{{9}}+D_{{5}}-D_{{22}},D_{{10}}-D_{{23}},-D_{{24}}+D_{{12}},-D_{{6}
}+D_{{13}}-D_{{25}},D_{{25}}-2\,D_{{13}}+D_{{1}},
D_{{7}},-D_{{26}},-D_
{{2}},-D_{{3}},-D_{{29}},D_{{5}}D_{{22}}-D_{{5}}^{2}-D_{{27}}+D_{{14
}},D_{{5}}D_{{23}}-D_{{28}}+D_{{16}},
-D_{{5}}D_{{24}}-D_{{13}}D_{{11}}
+D_{{17}},2\,D_{{5}}D_{{24}}+D_{{11}}D_{{25}}-D_{{30}},D_{{25}}D_{{5}}
-D_{{5}}D_{{13}}+D_{{8}},
D_{{25}}D_{{5}}-D_{{31}},-D_{{13}}D_{{25}}+D_{{13}}^{2}-D_{{4}},-D_{{33}},D_{{24}}D_{{5}}^{2}+D_{{25}}D_{{5}}D_{{11}}-D_{{35}},
-D_{{5}}^{2}D_{{22}}+D_{{5}}^{3}+D_{{5}}D_{{27}}-D_{{32}}+D_{{19}},
2\,D_{{5}}^{2}D_{{23}}-D_{{13}}D_{{15}}-D_{{5}}D_{{28}}+D_{{20}},
-3\,D_{{5}}^{2}D_{{23}}+D_{{15}}D_{{25}}+2\,D_{{5}}D_{
{28}}-D_{{34}},-D_{{36}},-2\,D_{{23}}D_{{5}}^{3}+$ $D_{{25}}D_{{5}}D_{{15}}+D_{{28}}D_{{5}}^{2}-D_{{38}},
-3\,D_{{5}}^{3}D_{{22}}+4\,D_{{5}}^{4}+2\,D_{{5}}^{2}D_{{27}}-D_{{13}}D_{{18}}-D_{{5}}D_{{32}}+D_{{21}},
4\,D_{{5}}^{3}D_{{22}}-5\,D_{{5}}^{4}-3\,D_{{5}}^{2}D_{{27}}+D_{{18}}D_{{25}}+2\,D_{{5}}D_{{32}}-D_{{37}},
3\,D_{{22}}D_{{5}}^{4}-4\,D_{{5}}^{5}-2\,D_{{27}}D_{{5}}^{3}+D_{{25}}D_{{5}}D_{{18}}+D_{{32}}D_{{5}}^{2}-D_{{39}})\subset K[D]$. 

After the  elimination  of the variables, there are again 12 free variables $D$ left.

By comparing the coefficients of the polynomials $f'_1+g'_1,f'_2+g'_2,f'_3+g'_3,f'_4+g'_4$ with those
of $p'_1,p'_2,p'_3,p'_4$, we can construct a $K$-algebra morphism $\phi$ between $K[D]/\mathfrak h'_0$ and $K[C]/\mathfrak h_0$. 
Since we are considering $K[D]/\mathfrak h'_0$, it is enough to give the images under $\phi$ of the 12 free variables $D$: 
\[\begin{array}{lll}
\phi(D_{{5}})=C_{{5}}, &\phi(D_{{11}})=C_{{11}},&\phi(D_{{13}})=C_{{6}}+C_{{1}},\\
\phi(D_{{15}})=C_{{15}},&\phi(D_{{18}})=C_{{18}},&\phi(D_{{22}})=C_{{5}}+C_{{9}},\\
\phi(D_{{23}})=C_{{10}}, &\phi(D_{{24}})=C_{{12}}, &\phi(D_{{25}})=-2\,C_{{6}}+C_{{1}},\\ \phi(D_{{27}})=C_{{5}}C_{{9}}+C_{{14}},&\phi(D_{{28}})=C_{{5}}C_{{10}}+C_{{16}},&\phi(D_{{32}})=C_{{5}}C_{{14}}+C_{{19}}
\end{array}\]
The morphism $\phi$ is exactly the one of Theorem \ref{isomorfismo}, hence it is an isomorphism between $K[D]/\mathfrak h'_0$ and $K[C]/\mathfrak h_0$.
\end{example}



\begin{example}
In this example, we compute radical liftings of a given homogeneous saturated ideal defining an aCM scheme of codimension two by the results of Section \ref{sec:acm}.

We consider the aCM   ideal  $H= (x_0^2-x_1^2, x_0x_1+2x_1^2, x_1^3)\subset K[x_0,x_1,x_2]$. The scheme in $\mathbb P^2$ defined by $H$ is non reduced and its support is a point, since the saturated ideal is not radical. The initial ideal of $H$ w.r.t.~degrevlex is  $\mathfrak j=(x_0^2,x_0x_1,x_1^3)$. In order to obtain a radical $x_3$-lifting of $H$ in $K[x_0,x_1,x_2,x_3]$ following the proof of Theorem \ref{thm:liftACMdue}, we first construct a radical $x_3$-lifting $N$ of $\mathfrak j$ obtained by a so-called distraction \cite{H66}: $N=(x_0(x_0+x_3), x_0x_1, x_1(x_1+x_3)(x_1-x_3))$. Observe that this ideal is not radical when $\mathrm{char}(K)=2$; in this case we will consider a different ideal $N$, at the end of this example.

We now write down the matrices associated to $\mathfrak j$, $H$ and $N$ and we consider the ideal $I$ which corresponds to the matrix $M_I=M_H +M_{N}-M_{\mathfrak j}$:

$$M_{\mathfrak j}=\left( \begin{array}{ccc}  x_1 & -x_0  & 0  \\ 0 & x_1^2 & -x_0 \end{array} \right)  
\quad \quad   M_H=\left( \begin{array}{ccc}  x_1 & -x_0+2x_1 & -3 \\ 0 & x_1^2 & -x_0-2x_1 \end{array} \right)  $$ 

$$ M_{N}=\left( \begin{array}{ccc}  x_1 & -x_0-x_3 & 0 \\ 0 & x_1^2-x_3^2 & -x_0 \end{array} \right) 
\quad M_{I}=\left( \begin{array}{ccc}  x_1 & -x_0+2x_1-x_3 & -3 \\ 0 & x_1^2-x_3^2 & -x_0-2x_1 \end{array} \right)$$
Note that the matrix $M_H$ determines the isomorphism $\psi_2$ of a free resolution of type \eqref{eq:acm resolution} which is not minimal, in this case.

The generators of the ideal $I$ are the minors of maximal order of the matrix $M_I$:
$$I=(x_0^2-x_1^2+x_0x_3+2x_1x_3-3x_3^2 , x_0x_1+2x_1^2, x_1^3-x_1x_3^2).$$
We can easily verify that $I$ is a radical lifting of $H$ assuming that $\mathrm{char}(K)\neq 13$: indeed, 
$$I= \left( x_0+2x_3, x_1-x_3)\cap (x_0+x_3, 2x_1-x_3)\cap (x_0^2+x_0x_3 -3x_3^2,x_1\right) $$
and the discriminant of $x_0^2+x_0x_3 -3x_3^2$ is $13$. 

On the other hand, if $\mathrm{char}(K)=13$, we consider the weight vector $\omega=[3,2,0,0]$ which makes every term in $J$ of degree 2 or 3 bigger than every term of the same degree in  $\cN(J)$. As in the proof of Theorem \ref{thm:liftACMdue}, we construct the ideal
$$I(t)=(x_0^2-x_1^2+t^{-3}x_0x_3+2t^{-3}x_1x_3-3t^{-4}x_3^2 , x_0x_1+2x_1^2, x_1^3-t^{-4}x_1x_3^2).$$
We replace $t$ by a  random integer, for instance $t=7$, and we obtain for $\mathrm{char}(K)=13$ the decomposition  $$I(7)=(x_0+2x_1, x_1-2x_3 )\cap (x_0+2x_1 , x_1+2x_3 )\cap ( x_0+x_3 , x_1 )\cap (x_0+4x_3 ,x_1).$$

When the field $K$ has characteristic $2$, we compute a $x_3$-lifting of $J$ assuming  $\vert K\vert\geq 3$ (for example, $K$ could be the algebraically closure of $\mathbb Z_2$) and letting $\chi$ be any element of $K$ different from $0,1$. Thus, we obtain the following radical $x_3$-lifting of $J$
$$N=(x_0(x_0+x_3),x_0x_1,x_1(x_1 + x_3)(x_1 + \chi x_3)) =\\  (x_0^2 + x_0x_3, x_0x_1, x_1^3 + (\chi+1) x_1^2x_3 + \chi x_1x_3^2).$$
In this case, the ideal $H$ becomes $(x_0^2+x_1^2,x_0x_1,x_1^3)$ and the matrices of syzygies are
$$M_H=\left( \begin{array}{ccc}  x_1 & x_0 & 1 \\ 0 & x_1^2 & x_0 \end{array} \right)  
\quad\quad
M_{N}=\left( \begin{array}{ccc}  x_1 & x_0+x_3 & 0 \\ 0 & x_1^2+(\chi+1)x_1x_3+\chi x_3^2 & x_0 \end{array} \right).$$
Anyway, in this case we take the matrix 
$M=\left( \begin{array}{ccc}  x_1 & x_0+x_3 & 1 \\ 0 & x_1^2+(\chi+1)x_1x_3+\chi x_3^2 & x_0 \end{array} \right)$
whose maximal minors define the following radical $x_3$-lifting of $H$:
$$\begin{array}{lll}
I&=&(x_1(x_1^2+(\chi+1)x_1x_3+\chi x_3^2),x_1x_0,x_0^2+x_0x_3+x_1^2+(\chi+1)x_1x_3+\chi x_3^2)=\\
&=&(x_1,\chi x_3^2+x_0^2+x_0x_3)\cap (x_0,\chi x_3+x_1)\cap (x_0,x_1+x_3).\end{array}$$
\end{example}


\begin{example} 
Consider the saturated monomial ideal $J=(x_0^2,x_0x_1,x_1^2) $ in $K[x_0,x_1,x_2]$, with $K$ of characteristic $0$. In this example, we study the locus of radical $x_3$-liftings of the ideal $J$. In order to construct the scheme of liftings of $J$ by Theorem \ref{FM} and Proposition \ref{comeLR}, we start from the polynomials
\[
\begin{array}{rcl}
f_1+g_1&=&x_0^{2}+C_{{1}}x_0x_3+C_{{2}}x_1x_3+C_{{3}}x_2x_3+C_{{4}}x_3^2, \\
f_2+g_2&=&x_0x_1+C_{{5}}x_0x_3+C_{{6}}x_1x_3+C_{{7}}x_2x_3+C_{{8}}x_3^2,\\
f_3+g_3&=&x_1^{2}+C_{{9}}x_0x_3+C_{{10}}x_1x_3+C_{{11}}x_2x_3+C_{{12}}x_3^2.
\end{array}
\]
Hence, we construct the scheme $\mathrm{L}_J$ as a subscheme of $\mathbb A^{12}$.
After  the elimination, we obtain that $\mathrm{L}_J$ is isomorphic to $\mathbb A^6$ and get the Gr\"obner basis 
\[
\begin{array}{rcl}
f_1+g_1&=&x_0^{2}+C_{1}x_0x_3+C_{2}x_1x_3+(-C_{5}C_{2}+C_{1}C_{6}+C_{2}C_{10}
-{C_{6}}^{2})x_3^2\\
f_2+g_2&=&x_0x_1-C_{{5}}x_0x_3+C_{{6}}x_1x_3+(C_{{6}}C_{{5}}-C_{{2}}C_{{9}})x_3^2\\
f_3+g_3&=&x_1^{2}+C_{{9}}x_0x_3+C_{{10}}x_1x_3+(C_{{10}}C_{{5}}+C_{{9}}C_{{1}}-C_{{6}}C_{{9
}}-{C_{{5}}}^{2})x_3^2.
\end{array}
\]
By definition of $x_3$-lifting, the hyperplane $x_3=0$ does not contain any irreducible component, nor isolated or embedded ones, of the scheme defined by a  $x_3$-lifting of $J$, hence we can put $x_3=1$ and work in the affine three-dimensional space.

If $C_2\neq0$, we obtain 
$x_1=-\frac{1}{C_2}x_0^2-\frac{C_{1}}{C_2}x_0+(\frac{C_6^2}{C_2}+C_5-\frac{C_{1}C_{6}}{C_2}-C_{10})$ \ from the polynomial $(f_1+g_1)(x_0,x_1,x_2,1)$ 
and replacing $x_1$ in $(f_2+g_2)(x_0,x_1,x_2,1)$ we get a degree $3$ polynomial
whose discriminant with respect to the variable $x_0$ is:
\[
\begin{array}{rcl}
\Delta&=&
16\,C_{{10}}C_{{1}}C_{{5}}C_{{6}}+18\,C_{{2}}C_{{9}}C_{{1}}C_{{10}}-36
\,C_{{2}}C_{{9}}C_{{1}}C_{{5}}-4\,{C_{{1}}}^{2}C_{{10}}C_{{5}}-48\,C_{
{9}}C_{{1}}{C_{{6}}}^{2}+\\
& &-16\,C_{{5}}C_{{10}}{C_{{6}}}^{2}-36\,C_{{2}}C
_{{10}}C_{{9}}C_{{6}}-4\,{C_{{10}}}^{2}C_{{1}}C_{{6}}+72\,C_{{5}}C_{{2
}}C_{{9}}C_{{6}}-16\,{C_{{5}}}^{2}C_{{1}}C_{{6}}+\\
& &+24\,C_{{9}}{C_{{1}}}^
{2}C_{{6}}+{C_{{1}}}^{2}{C_{{10}}}^{2}+4\,{C_{{1}}}^{2}{C_{{5}}}^{2}-4
\,{C_{{1}}}^{3}C_{{9}}+24\,C_{{5}}{C_{{10}}}^{2}C_{{2}}+\\
& &-48\,{C_{{5}}}^{2}C_{{10}}C_{{2}} -27\,{C_{{2}}}^{2}{C_{{9}}}^{2}+32\,{C_{{5}}}^{3}C_{
{2}}-4\,C_{{2}}{C_{{10}}}^{3}+32\,C_{{9}}{C_{{6}}}^{3}+16\,{C_{{5}}}^{
2}{C_{{6}}}^{2}+\\
& & +4\,{C_{{10}}}^{2}{C_{{6}}}^{2}.
\end{array}
\]
If $C_9\neq 0$, the analogous argument applied on $(f_3+g_3)(x_0,x_1,x_2,1)$ used to eliminate $x_0$ leads to the same discriminant $\Delta$. Finally, assuming that  $C_2=C_9=0$ gives a polynomial $(f_2+g_2)(x_0,x_1,x_2,1)$ that decomposes as $(x_0+C_6)(x_1-C_5)$. We study   separately  the two components corresponding to these two factors.  Replacing $x_0$ by $-C_6$ we find that $f_2+g_2$, and  $f_3+g_3$  vanish, while $f_1+g_1$ becomes a degree 2  polynomial in $x_0$ with discriminant $(C_1-2C_6)^2$. Analogously, replacing  $x_1$ by $C_5$  we get $(C_{10}+2C_5)^2$ as the discriminant of   $f_3+g_3$.  Then,   the locus of non-radical $x_3$-liftings  is  defined  by $\Delta$ also when  $C_2=C_9=0$, since   $\Delta =(C_1-2C_6)^2(C_{10}+2C_5)^2$ in $K[C]/(C_2,C_9)$.

Summing up, the locus of radical liftings of $J$ is the open subset which is the complement of the closed one defined by $(\Delta)$.
\end{example}

\bcR
The following last example shows that a scheme of $x_n$-liftings can be non-reduced.
 
\begin{example}\label{es:struttNonRid}
Let $H \subset K[x_0,\ldots,x_3]=K[\mathbf x]$ be the ideal generated by the reduced Gr\"obner basis $\{x_0x_3^3,x_0^3,x_0^2x_1,x_0x_1^2,x_1^3,$ $x_0^2x_2,x_0x_2^2,x_2^3-x_3^3\}$, w.r.t.~the degrevlex order on $K[\mathbf x]$. We construct the scheme of the $x_4$-liftings of $H$ starting from the set $\mathcal G=\{f_i+g_i\}_{i\in \{1,\ldots,8\}}$, where $f_i$ is  the $i$-th polynomial in the list of generators of $H$, and $g_i=\sum_j C_{i,j}\tau_j$ as in \eqref{code}. 

We obtain that the ideal $\mathfrak h_0$ (as in Definition \ref{defh0}) that defines the scheme of $x_4$-liftings is contained in $K[C]$, with $\vert C\vert =\vert\{C_{i,j}\}\vert=133$. After the elimination of $100$  variables, we obtain the ideal $\overline{\mathfrak h_0}$ defining $\mathrm L_H$ as a subscheme of its tangent space at $H$ in a polynomial ring in the remaining $33$ variables $\overline C$. In the reduced  Gr\"obner basis  of $\overline{\mathfrak h_0}$ we find some non-reduced polynomials, as for instance $C_{1, 3}^2(3C_{4, 2}-2C_{5, 3})$. Then, we can conclude that  $\overline{\mathfrak h_0}$ is a non-radical ideal, that is $\mathrm L_H$ is a non-reduced scheme. 
\end{example}
\ecr

\section*{Acknowledgements}
The second and fourth authors were partially supported by the PRIN 2010-11  \emph{Geometria delle variet\`a algebriche}, cofinanced by MIUR (Italy). We thank  Davide Franco for kindly providing useful background about the locus of points of a Hilbert scheme corresponding to reduced schemes.

\bibliographystyle{amsplain}

\begin{thebibliography}{10}

\bibitem{CoCoA-5}
J.~Abbott, A.M. Bigatti, and G.~Lagorio, \emph{{CoCoA-5}: a system for doing
  {C}omputations in {C}ommutative {A}lgebra}, Available at
  \texttt{http://cocoa.dima.unige.it}.

\bibitem{AM}
M.~F. Atiyah and I.~G. Macdonald, \emph{Introduction to commutative algebra},
  Addison-Wesley Publishing Co., Reading, Mass.-London-Don Mills, Ont., 1969.

\bibitem{BGS91}
D.~Bayer, A.~Galligo, and M.~Stillman, \emph{Gr\"obner bases and extension of
  scalars}, Computational algebraic geometry and commutative algebra
  ({C}ortona, 1991), Sympos. Math., XXXIV, Cambridge Univ. Press, Cambridge,
  1993, pp.~198--215.

\bibitem{BM91}
D.~Bayer and D.~Mumford, \emph{What can be computed in algebraic geometry?},
  Computational algebraic geometry and commutative algebra ({C}ortona, 1991),
  Sympos. Math., XXXIV, Cambridge Univ. Press, Cambridge, 1993, pp.~1--48.

\bibitem{BS2}
D.~Bayer and M.~Stillman, \emph{A criterion for detecting {$m$}-regularity},
  Invent. Math. \textbf{87} (1987), no.~1, 1--11.

\bibitem{BB73}
A.~Bia{\l}ynicki-Birula, \emph{Some theorems on actions of algebraic groups},
  Ann. of Math. (2) \textbf{98} (1973), 480--497.

\bibitem{BB76}
\bysame, \emph{Some properties of the decompositions of algebraic varieties
  determined by actions of a torus}, Bull. Acad. Polon. Sci. S\'er. Sci. Math.
  Astronom. Phys. \textbf{24} (1976), no.~9, 667--674.

\bibitem{BuchEis}
D.~A. Buchsbaum and D.~Eisenbud, \emph{On a problem in linear algebra},
  Conference on {C}ommutative {A}lgebra ({U}niv. {K}ansas, {L}awrence, {K}an.,
  1972), Springer, Berlin, 1973, pp.~50--56. Lecture Notes in Math., Vol. 311.

\bibitem{CF}
G.~Carr{\`a}~Ferro, \emph{Gr\"obner bases and {H}ilbert schemes. {I}}, J.
  Symbolic Comput. \textbf{6} (1988), no.~2-3, 219--230, Computational aspects
  of commutative algebra.

\bibitem{CFRo}
G.~Carr{\`a}~Ferro and L.~Robbiano, \emph{On super {G}-bases}, J. Pure Appl.
  Algebra \textbf{68} (1990), no.~3, 279--292.

\bibitem{CMR}
M.~Ceria, T.~Mora, and M.~Roggero, \emph{Term-ordering free involutive bases},
  In press on Journal of Symbolic Computation. DOI:10.1016/j.jsc.2014.09.005,
  2014.

\bibitem{CV}
A.~Conca and G.~Valla, \emph{Canonical {H}ilbert-{B}urch matrices for ideals of
  {$k[x,y]$}}, Michigan Math. J. \textbf{57} (2008), 157--172, Special volume
  in honor of Melvin Hochster.

\bibitem{Co}
A.~Constantinescu, \emph{Parametrizations of ideals in {$K[x,y]$} and
  {$K[x,y,z]$}}, J. Algebra \textbf{346} (2011), 1--30.

\bibitem{Ei}
D.~Eisenbud, \emph{Commutative algebra}, Graduate Texts in Mathematics, vol.
  150, Springer-Verlag, New York, 1995, With a view toward algebraic geometry.

\bibitem{Elli}
G.~Ellingsrud, \emph{Sur le sch\'ema de {H}ilbert des vari\'et\'es de
  codimension {$2$} dans {${\bf P}^{e}$} \`a c\^one de {C}ohen-{M}acaulay},
  Ann. Sci. \'Ecole Norm. Sup. (4) \textbf{8} (1975), no.~4, 423--431.

\bibitem{FR}
G.~Ferrarese and M.~Roggero, \emph{Homogeneous varieties for {H}ilbert
  schemes}, Int. J. Algebra \textbf{3} (2009), no.~9-12, 547--557.

\bibitem{FlR}
G.~Fl{\o}ystad and M.~Roggero, \emph{Borel degenerations of arithmetically
  {C}ohen-{M}acaulay curves in ${\Bbb{p}}^3$}, Internat. J. Algebra Comput.
  \textbf{24} (2014), no.~5, 715--739.

\bibitem{Fo}
J.~Fogarty, \emph{Algebraic families on an algebraic surface}, Amer. J. Math
  \textbf{90} (1968), 511--521.

\bibitem{GGR}
A.~V. Geramita, D.~Gregory, and L.~Roberts, \emph{Monomial ideals and points in
  projective space}, J. Pure Appl. Algebra \textbf{40} (1986), no.~1, 33--62.

\bibitem{Go88}
G.~Gotzmann, \emph{A stratification of the {H}ilbert scheme of points in the
  projective plane}, Math. Z. \textbf{199} (1988), no.~4, 539--547.

\bibitem{Gro}
A.~Grothendieck, \emph{\'{E}l\'ements de g\'eom\'etrie alg\'ebrique. {IV}.
  \'{E}tude locale des sch\'emas et des morphismes de sch\'emas {IV}}, Inst.
  Hautes \'Etudes Sci. Publ. Math. (1967), no.~32, 361.

\bibitem{H66}
R.~Hartshorne, \emph{Connectedness of the {H}ilbert scheme}, Inst. Hautes
  {\'E}tudes Sci. Publ. Math. (1966), no.~29, 5--48.

\bibitem{H77}
\bysame, \emph{Algebraic geometry}, Springer-Verlag, New York-Heidelberg, 1977,
  Graduate Texts in Mathematics, No. 52.

\bibitem{KR1}
M.~Kreuzer and L.~Robbiano, \emph{Computational commutative algebra. 1},
  Springer-Verlag, Berlin, 2000.

\bibitem{KR2}
\bysame, \emph{Computational commutative algebra. 2}, Springer-Verlag, Berlin,
  2005.

\bibitem{Led}
M.~Lederer, \emph{Gr\"obner strata in the {H}ilbert scheme of points}, J.
  Commut. Algebra \textbf{3} (2011), no.~3, 349--404.

\bibitem{LR}
P.~Lella and M.~Roggero, \emph{Rational components of {H}ilbert schemes},
  Rendiconti del Semi-nario Matematico dell'Universit\`a di Padova \textbf{126}
  (2011), 11--45.

\bibitem{LR2}
\bysame, \emph{On the functoriality of marked families}, Available at
  arXiv:1307.7657 [math.AG], 2013.

\bibitem{LY}
T.~Luo and E.~Yilmaz, \emph{On the lifting problem for homogeneous ideals}, J.
  Pure Appl. Algebra \textbf{162} (2001), no.~2-3, 327--335.

\bibitem{M}
F.~S. Macaulay, \emph{Some properties of enumeration in the theory of modular
  systems}, Proc. London Math. Soc. (1926), no.~26, 531--555.

\bibitem{Mall00}
D.~Mall, \emph{Connectedness of {H}ilbert function strata and other
  connectedness results}, J. Pure Appl. Algebra \textbf{150} (2000), no.~2,
  175--205.

\bibitem{Maple}
Maplesoft, \emph{Maple - technical computing software for engineers,
  mathematicians, scientists, instructors and students}, 2012.

\bibitem{MiNa}
J.~Migliore and U.~Nagel, \emph{Lifting monomial ideals}, Comm. Algebra
  \textbf{28} (2000), no.~12, 5679--5701, Special issue in honor of Robin
  Hartshorne.

\bibitem{SPES2}
T.~Mora, \emph{Solving polynomial equation systems. {II}}, Encyclopedia of
  Mathematics and its Applications, vol.~99, Cambridge University Press,
  Cambridge, 2005, Macaulay's paradigm and Gr{\"o}bner technology.

\bibitem{NS}
R.~Notari and M.~L. Spreafico, \emph{A stratification of {H}ilbert schemes by
  initial ideals and applications}, Manuscripta Math. \textbf{101} (2000),
  no.~4, 429--448.

\bibitem{Rlg}
Leslie~G. Roberts, \emph{On the lifting problem over an algebraically closed
  field}, C. R. Math. Rep. Acad. Sci. Canada \textbf{11} (1989), no.~1, 35--38.

\bibitem{RT}
M.~Roggero and L.~Terracini, \emph{Ideals with an assigned initial ideals},
  Int. Math. Forum \textbf{5} (2010), no.~53-56, 2731--2750.

\bibitem{Ro}
M.~Roitman, \emph{On the lifting problem for homogeneous ideals in polynomial
  rings}, J. Pure Appl. Algebra \textbf{51} (1988), no.~1-2, 205--215.

\end{thebibliography}

\end{document}